%% file: multiscale.tex
\documentclass{article}

\usepackage{amsmath,amssymb,enumerate,bbm,mathrsfs}
\usepackage{epsfig}

\newtheorem{cntr}{ERROR! SHOULD NOT USE THIS}

\newcommand{\tmop}[1]{\operatorname{#1}}
\newtheorem{definition}[cntr]{Definition}
\newtheorem{assumption}{Assumption}
\newcommand{\dueto}[1]{\textup{\textbf{(#1) }}}
\newtheorem{varremark}[cntr]{Remark}
\newenvironment{remark}{\begin{varremark}\em}{\em\end{varremark}}
\newcommand{\nin}{\not\in}

\newtheorem{varnote}{Note}

\newtheorem{proposition}[cntr]{Proposition}

\newenvironment{proof}{
  \noindent\textbf{Proof.}\ }{\hspace*{\fill}
  \begin{math}\Box\end{math}\medskip}

\newenvironment{proofof}[1]{
  \noindent\textbf{Proof of #1.}\ }{\hspace*{\fill}
  \begin{math}\Box\end{math}\medskip}

\newenvironment{proof*}[1]{
  \noindent\textbf{#1\ }}{\hspace*{\fill}
  \begin{math}\Box\end{math}\medskip}

\newtheorem{lemma}[cntr]{Lemma}

\newtheorem{theorem}[cntr]{Theorem}

\newtheorem{algo}{Algorithm}

\newcommand{\constfont}[1]{\mathbf{#1}}

\newcommand{\FFT}[1]{\mathcal{F}_{#1}}
\newcommand{\FFTi}[1]{\mathcal{F}^{-1}_{#1}}
\newcommand{\SI}[1]{\mathcal{I}_{#1}}

\newcommand{\ks}[0]{k_0}

\newcommand{\Zn}[0]{{\mathbb{Z}^N}}
\newcommand{\Rn}[0]{{\mathbb{R}^N}}

\newcommand{\XBuffCube}[3]{\constfont{X}}
\newcommand{\KBuffCube}[2]{\constfont{K}}

\newcommand{\frmltbbDef}[3]{ UNDEFINED AS OF YET}

\newcommand{\kmax}[0]{k_{\tmop{max}}}
\newcommand{\kmin}[0]{k_{\tmop{min}}}
\newcommand{\Tstep}[0]{T_{\tmop{step}}}
\newcommand{\Tmax}[0]{T_{\tmop{max}}}

\newcommand{\InteractionRegion}[0]{\mathcal{I}_{\kmax,\kmin}}
\newcommand{\IntBox}[0]{\mathcal{I}}

\newcommand{\norm}[2]{\left\| #1 \right \|_{#2}}
\newcommand{\abs}[1]{\left| #1 \right|}
\newcommand{\absSmall}[1]{| #1 |}

\newcommand{\Hs}[0]{H^s}
\newcommand{\Hsb}[0]{H^s_b}

\newcommand{\jap}[1]{\langle #1 \rangle}

\newcommand{\erf}[0]{\tmop{erf}}
\newcommand{\erfc}[0]{\tmop{erfc}}
\newcommand{\sinc}[0]{\tmop{sinc}}

\newcommand{\Lap}[0]{(1/2)\Delta}

\numberwithin{equation}{section}
\numberwithin{cntr}{section}

\newcommand{\comment}[1]{}

\begin{document}

\title{Multiscale Resolution of Shortwave-Longwave Interaction}

\author{A. Soffer and C. Stucchio}

\maketitle

\begin{abstract}
  In the study of time-dependent waves, it is computationally expensive to solve a problem in which high frequencies (shortwaves, with wavenumber $k = \kmax$) and low frequencies (longwaves, near $k=\kmin$) mix. Consider a problem in which low frequencies scatter off a sharp impurity. The impurity generates high frequencies which propagate and spread throughout the computational domain, while the domain must be large enough to contain several longwaves. Conventional spectral methods have computational cost proportional to $O(\kmax/\kmin \log (\kmax/\kmin))$.

We present here a multiscale algorithm (implemented for the Schr\"odinger equation, but generally applicable) which solves the problem with cost (in space and time)  $O(\kmax L \log(\kmax / \kmin) \log(\kmax L))$. Here, $L$ is the width of the region in which the algorithm resolves all frequencies, and is independent of $\kmin$.
\end{abstract}

\section{Introduction and Definitions}

Consider the time dependent Schr\"odinger equation with $(x,t) \in \mathbb{R}^{N+1}$:
\begin{subequations}
  \label{eq:schroIVP}
  \begin{equation}
    \label{eq:schro}
    i \partial_{t} \psi(x,t) = \left[-(1/2)\Delta +V(x)\right]\psi(x,t)
  \end{equation}
  \begin{equation}
    \psi(x,t=0)=\psi_{0}(x)
  \end{equation}
\end{subequations}
In this work we restrict ourselves to the case $N=1$, although there is nothing intrinsic to our method which requires this.

Suppose that for the chosen initial data, $\hat{\psi}(k,t)$ remains localized between the frequencies $\kmin$ and $\kmax$ (with $\kmin \ll \kmax$), apart from some small error.

To solve this numerically, one typically truncates the domain to a finite region (the interaction region) and imposes some sort of open boundary conditions. The natural question to ask at this point is ``what is the interaction region?''

Assume $V(x)$ is supported near $x=0$. Waves have an interaction length proportional to their wavelength, implying that waves with wavelength $\lambda = 2\pi/k$ interact with $V(x)$ over a distance $O(\lambda)=O(k^{-1})$. Thus, if frequencies are bounded below by $\kmin$, the interaction region has width at least $O(\kmin^{-1})$. Accurate resolution of high frequency waves requires at least two samples per period of the highest frequency $\kmax$, or a sampling rate $O(\kmax)$. The number of samples is $O(\kmax/\kmin)$ (sampling rate times box width). In phase space terms, the solution is being computed on the region
\begin{equation}
  \left\{ (x,k) : \abs{x} \leq \frac{C}{\kmin}, \abs{k} \leq \kmax \right\},
\end{equation}
and the number of samples required is proportional to the phase space volume of this region.

However, high frequency waves (with, e.g. $k=\kmax/2$) do not interact with the potential past $x=O(k^{-1})$. So the interaction region in phase space is:
\begin{equation}
  \label{eq:6}
  \IntBox = \{ (x,k) : \abs{k} \leq C/\abs{x}, \abs{k} \leq \kmax \}.
\end{equation}
Assuming $\hat{\psi}(k,t)$ contains little mass in $[-\kmin,\kmin]$, we can truncate at $x=O(\kmin^{-1})$. This truncated region has volume $O(\kmax \log(\kmin^{-1}))$ rather than\\$O(\kmax \kmin^{-1})$, which is asymptotically smaller (the constant of proportionality is proportional to $C^{1/2}$ and has units of length). This observation suggests that ordinary spectral methods are inefficient for studying interactions, and that more efficient numerical methods are possible.

In this work we present an algorithm exploiting this with spectral accuracy. The computational complexity is $O(L \kmax \log (\kmax/\kmin) \log (L \kmax))$ per timestep, rather than $O(\kmax/\kmin)$ for normal spectral methods. Additionally, being a Fourier-based spectral method, it can be easily modified to treat other wave equations. Here, $L$ is chosen so that the high frequency components of $\IntBox$ are entirely contained $[-L,L]$ (i.e. if $(x,\kmax/2) \in \IntBox$, then $x \in [-L,L]$). Note that $L$ is independent of $\kmin$ (though it increases with  the size of $V(x)$).

The method presented here can be applied to slowly decaying potentials, requiring only that $V(x)$ is smooth and $\abs{V(x)} \leq C/\jap{x}^{2}$. This rate of decay is the slowest decay rate preserving the shape of $\IntBox$.

\subsection{Heuristics and Intuition}

In this section, we heuristically explain our decay conditions on $V(x)$. In particular, we sketch an argument why $\abs{V(x)} \leq  C/\jap{x}^{2}$ allows the interaction region to take the form \eqref{eq:6}, but $C$ now varies depending on the potential (rather than simply Heisenberg localization). We roughly follow the construction of the half-wave parametrix found in \cite{sogge:fiobook}, and assume $V(x)$ is smooth and decays at the rate $\abs{V(x)} \leq C/ \jap{x}^{2}$.

To begin, let $\phi(x,k)$ solve the (classical) Hamilton-Jacobi equation
\begin{subequations}
  \label{eq:kineticEquations}
  \begin{equation}
    \label{eq:hamiltonJacobi}
    \frac{1}{2}\left[ \partial_{x} \phi(x,y,k) \right]^{2} + V(x) = \frac{1}{2} k^{2} + V(y)
  \end{equation}
  \begin{equation}
    \label{eq:hamiltonJacobiInitialCondition}
    \partial_{x} \phi(x,y,k) = k ~\textrm{when}~ x=y
  \end{equation}
  The solution to \eqref{eq:hamiltonJacobi} is not single valued everywhere, though it is single valued on outgoing trajectories. We then let $q_{0}(t,x,y,k)$ solve the classical transport equation (valid only where $\phi(x,y,k)$ is single valued):
  \begin{equation}
    \label{eq:TransportEquation}
    \partial_{t} q_{0} = k \partial_{x} q_{0} + \left[\partial_{x} V(x) \right] \partial_{k} q_{0}
  \end{equation}
  \begin{equation}
    \label{eq:TransportEquationInitialCondition}
    q_{0}(x,y,k) = I(x,y,k)
  \end{equation}
  with $I(x,y,k)$ a smooth function supported in a small neighborhood (a Heisenberg cell) around $(x,y,k) =(x_{0},x_{0},k_{0})$.
\end{subequations}
Subject to these conditions, it is shown in \cite[Chapter 4]{sogge:fiobook} that:
\begin{multline}
  \label{eq:3}
  e^{i H t} I(x,y,k)  - \int e^{i \left[\phi(x,y,k)+t(k^{2}+V(y)) \right]} q_{0}(t,x,y,k) dk\\
  = O(k^{-2} V(x)) = O\left( \frac{1}{k^{2} x^{2}} \right)
\end{multline}

The key to making sense of this is the observation that the characteristics of the Transport equation \eqref{eq:TransportEquation} are the classical trajectories associated to the classical Hamiltonian $(1/2)p^{2}+V(x)$. Thus, $e^{i H t}$ moves electrons along classical trajectories, ignoring quantum effects.

It is easily seen that a classical trajectory beginning at $(x_{0},k_{0})$ is outgoing if $k_{0}^{2}/2 >  C / \jap{x_{0}}^{2} > \abs{V(x_{0})}$ and $k_{0} x_{0} > 0$ (the momentum is pointed away from the origin). In one dimension, we can simply compute the phase function:
\begin{equation*}
  \phi(x,y,k) = \int_{x_{0}}^{x} \sqrt{k^{2}+2V(y)-2V(x')} dx'
\end{equation*}
Since $k \geq C/\jap{x_{0}}$, and $\abs{V(y)-V(x')} \leq C/\jap{x_{0}}^{2}$ (for different $C$), we find the term under the square root is always positive, and therefore single valued.

Thus, trajectories originating in the set $\{ (x_{0},k_{0}) : \abs{k_{0}} > C/\jap{x_{0}}, x_{0} k_{0} > 0 \}$ are outgoing, and it therefore makes sense to filter waves from this region. Provided we consider waves well within this region (i.e. make $C$ sufficiently large), quantum effects can not spread the outgoing trajectories too much.

This argument fails for $(x_{0},k_{0})$ in $\IntBox$. In this region, $\phi(x,y,k)$ becomes multivalued and non-simple. This is the region we wish to resolve numerically. If $V(x)$ decays more quickly than $O(x^{-2})$ we make no gain because even though the potential is not spread out, low frequency waves are and therefore the interaction region still looks like \eqref{eq:6}.

If $V(x)$ decays more slowly than $O(x^{-2})$, the shape of $\IntBox$ changes. For instance, if $\abs{V(x)} \leq C/\jap{x}$, then $\IntBox = \{ (x,k) : \abs{k} \leq C/ \jap{x}^{1/2} \}$. We believe that our method can be extended to such cases, but neglect this to keep things simple (see also Section \ref{sec:fftAlgorithmPossibleImprovement}).

If $k$ is restricted to the interval $[\kmin, \kmax]$, the volume of $\IntBox$ behaves like $O(\kmax \log (\kmin^{-1}))$ as $\kmin \rightarrow 0$. As remarked earlier, covering $\IntBox$ with a single rectangle (as is done in typical spectral methods), requires solving the problem on a region with volume $O(\kmax/\kmin)$. This volume is much larger than the interaction region, and wastes computation time and space.

The algorithm we construct here involves covering $\InteractionRegion$ with phase space rectangles of the form $[-2^{m} L,2^{m}L] \times [-2^{-m}\kmax,2^{-m}\kmax]$.  More precisely, on the region $[-2^{-m}L,2^{m}L]$, we do calculations on a lattice with lattice spacing $2\pi/(2^{-m}\kmax)$ allowing the accurate resolution of frequencies up to $2^{-m}\kmax$. This allows the computational volume to behave asymptotically like the interaction region (up to a constant). See Figure \ref{fig:coveringsOfPhaseSpace} for an illustration.

The number $L$ is chosen large enough so that waves with frequency greater than $2^{-m-1} \kmax$ are outgoing by the time they reach $x=2^{m} L/2$. In practice, we generally want the kinetic energy to be at least $5-10$ times the potential energy by the time we filter. For $\abs{V(x)} < V_{0} x^{-2}$, this requires $(\kmax/2)^{2}/2 \geq V_{0} L^{-2}$ or $L \geq \kmax^{-1} \sqrt{8 V_{0}/5}$ (see Figure \ref{fig:phaseSpaceRegions} which illustrates this criteria). This will ensure that the phase space filters remove primarily outgoing waves. For the interior propagator to work, we also require that $\L \kmax = O( \ln(\delta_{1}^{-1}))$ with $\delta_{1}$ a desired error (this ignores logarithmic prefactors; see \eqref{eq:ConstraintsOnSigmaKmaxL} for the details). Last but not least, $L$ must be sufficiently large so that $[-L/2,L/2]$ encompasses the region of interest, i.e. the region where we wish to resolve all frequencies.

\begin{remark}
  In spite of many attempts, \eqref{eq:kineticEquations} makes a poor basis for a numerical method except in the limit $k \rightarrow \infty$ (see, e.g. \cite{osher:schrodinger}). The reason for this is that when $\phi(x,y,k)$ becomes multivalued, most techniques break down. As far as the authors are aware, no numerical scheme performs better than operator splitting spectral methods in regions of phase space with complicated interactions.
\end{remark}

\subsection{The Interior Solver}

The interior solver is based on the split step method. In the usual split step method, the approximation
\begin{equation}
  \label{eq:splitStepIntro}
  e^{-i(-\Lap+V(x))t} \approx e^{i \Lap \delta t/2}\left[ \prod_{j=0}^{t/\delta t}  e^{-i V(x) t} e^{i \Lap \delta t} \right] e^{-i \Lap \delta t/2}
\end{equation}
is used. The operator $e^{i \Lap \delta t}$ is approximated by Fourier transforming the wavefunction (using the FFT, or Fast Fourier Transform), multiplying by $e^{i k^{2}/2 \delta t}$ and inverse Fourier transforming. Of course, the standard Fourier transform is based on localizing $\psi(x,t)$ in a rectangular region of phase space.

Our interior propagator is similarly based on \eqref{eq:splitStepIntro}, but the frequency domain operators are computed differently. Instead of using a single rectangular region of phase space, we cover the interaction region by a union of rectangles (c.f. Figure \ref{fig:coveringsOfPhaseSpace}). That is, we take a region $[-L,L]$ and use sample spacing $\delta x$ to represent the region $[-L,L] \times [-\kmax,\kmax]$ in phase space. We simultaneously study the regions $[-2^{m}L,2^{m}L] \times [-2^{-m}\kmax,2^{-m}\kmax]$ with lattice spacing $2^{m} \delta x$. To deal with the non-uniform sampling in $x$, we decompose the wavefunction into a sum of pieces, each of which is uniformly sampled and localized in a single rectangle in phase space. We then use ordinary spectral methods to apply the frequency domain operators to these pieces, and reconstruct the wavefunction by (spectrally accurate) interpolation afterward.

Like the regular FFT, our algorithm achieves spectral convergence. By spectral convergence, we mean that the all the error comes from parts of the wavefunction located outside the region of phase space we are considering. We provide a rigorous proof of it's accuracy\footnote{We neglect floating point errors, and all errors associated to the FFT besides aliasing in $x$ and $k$. These are small in practice, so we believe this is reasonable to do.} in Appendix \ref{sec:ProofOfCorrectness}. Combining this algorithm with standard operator splitting yields a method comparable to normal spectral methods. The complexity of the algorithm is $O(M \kmax L \log(\kmax L))$, with $M = O(\log (\kmin^{-1}))$ the number of dyadic scales needed.

The dominant part of the error comes from low frequencies, for which $k^{-1}$ is too large to resolve even on the coarsest grid (having width $2^{M}2L$). However, ``finite'' propagation speed\footnote{The Schr\"odinger equation does not have finite propagation speed. However, if frequencies are bounded above by $\kmax$, then velocities are bounded by $\kmax$ as well.} combined with virial identities suggest problems associated to low frequencies will take an exponentially long time (in the number of scales $M$) to appear; see Section \ref{sec:interpretingTheError}.

The approach given in this paper is specialized to the case where the wave operator is $-\Lap$, or $\omega(k)=(1/2)k^{2}$ in $1$ space dimension, though generalizations are straightforward. Provided operator splitting is valid, the method described here can be used. Of course, the rate of downscaling must be sufficient to completely capture the interaction region (see Section \ref{sec:fftAlgorithmPossibleImprovement}).

\begin{figure}
\setlength{\unitlength}{0.240900pt}
\ifx\plotpoint\undefined\newsavebox{\plotpoint}\fi
\sbox{\plotpoint}{\rule[-0.200pt]{0.400pt}{0.400pt}}%
\includegraphics[scale=0.6]{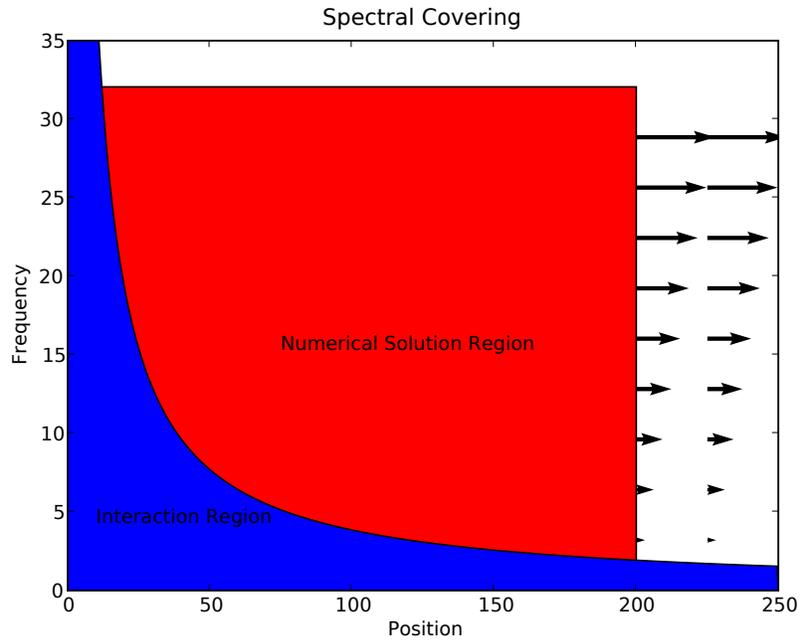}
\includegraphics[scale=0.6]{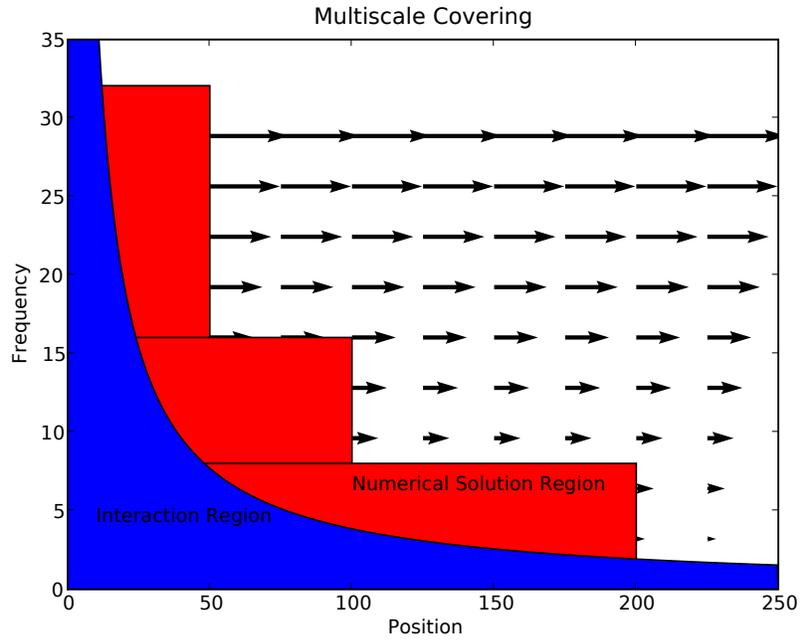}
\caption{A schematic illustration of the phase space regions where different algorithms work. The arrows indicate the direction of the flow.}
\label{fig:coveringsOfPhaseSpace}
\end{figure}

\subsection{Phase Space Filters for Outgoing Waves}

It should be noted that waves will almost always propagate outside the interaction region; we are just uninterested in them because their behavior is well understood. We must therefore come up with a means of removing the outgoing waves before they introduce errors into the numerical solution. This is in sharp contrast to  parabolic equations, where the dynamics of the equation dissipates high frequencies (which is the basis for multigrid techniques).

Since the computational boundary is defined in phase space rather than simply position, all the usual methods of open boundaries \cite{MR1913093,MR596431,MR0436612,MR517938,MR2060329,szeftel:absorbingBoundaries,colonius:artificialBoundaries} do not apply. Dirichlet-to-Neumann boundary conditions must be imposed at a particular curve in $x$. Artificial dissipation \cite{MR1294924,neuhauser:complexPotentials} terms like complex potentials or the PML will either fail to dissipate high frequency outgoing waves (if placed near $x=C/\kmin$) or will incorrectly dissipate low frequency waves which are still interacting with the potential (if placed near $x=C/\kmax$).

The only approach we are aware of which can be generalized to non-rectangular regions of phase space is the Time Dependent Phase Space Filter (TDPSF) \cite{us:TDPSFjcp}. The TDPSF algorithm (in one dimension, on a rectangular region of phase space) is as follows.

First, we solve \eqref{eq:schroIVP} on the finite region $[-L,L]$. Inside the regions $[-L,-L/2]$ and $[L/2,L]$, we periodically in time apply a phase space filter (with period $\Tstep$). We decompose $\psi(x,n \Tstep)=\psi_{out}(x)+\psi_{R}(x)$. The piece $\psi_{out}(x)$ is strictly outgoing, being localized on the region $[L/2,L] \times [\kmin,\kmax]$ in phase space. The remainder consists of waves located inside the region $[-L/2,L/2]$, as well as low frequency waves which may be located in $[-L,-L/2] \cup [L/2,L]$.

Here, ``low frequencies'' refers to frequencies smaller than $\kmin = O((L/2)^{-1})$. Since our filtering region has width (in position) $L/2$, we can localize in frequency no better than $O((L/2)^{-1})$.

For our problem, with a non-rectangular region of phase space, we wish to do the following instead. We will filter off waves in the region $[2^{m}L/2,2^{m}L] \times [2^{-m} \kmax/2, 2^{-m}\kmax]$ for $m=1 \ldots M$ (for some integer $M$). Figure \ref{fig:phaseSpaceRegions} diagrams the interaction region, filter regions and lattice spacing used.

\begin{figure}
\setlength{\unitlength}{0.240900pt}
\ifx\plotpoint\undefined\newsavebox{\plotpoint}\fi
\sbox{\plotpoint}{\rule[-0.200pt]{0.400pt}{0.400pt}}%
\includegraphics[scale=0.65]{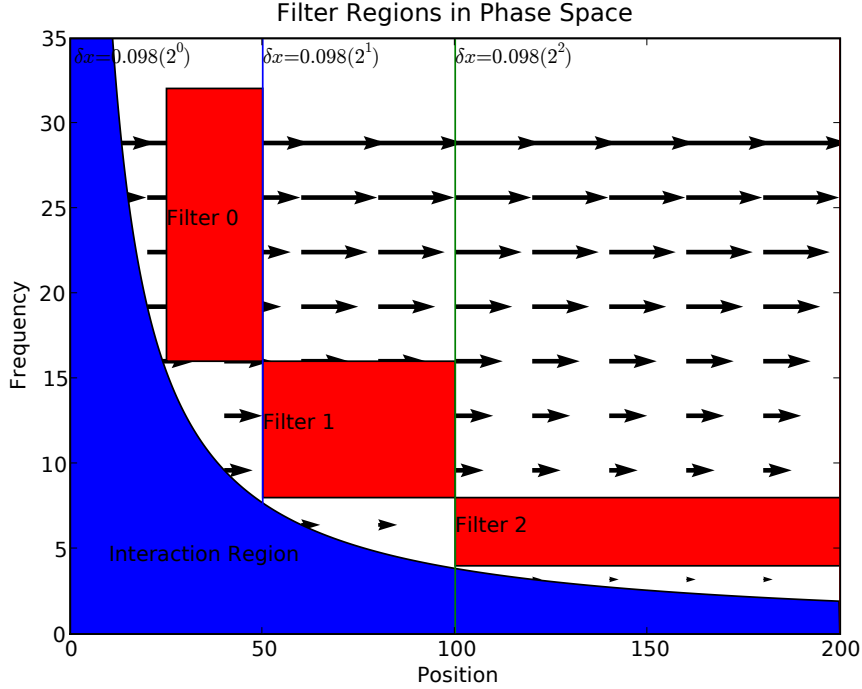}
\caption{The regions of phase space where the filter acts (with 3 scales). The arrows indicate the amplitude and direction of motion. The region of phase space in which waves interact with the potential is marked. Regions of phase space with $k \leq 0$ correspond to incoming waves, and should not be filtered. }
\label{fig:phaseSpaceRegions}
\end{figure}

\subsection{The Main Algorithm}

The general framework of the Multiscale TDPSF propagation algorithm can now be described in general terms; the specific details of the internal propagator and outgoing wave filter will be treated in Sections \ref{sec:multiscaleCalculationDO} and \ref{sec:phaseSpaceFilters} respectively.

\begin{algo}
  \dueto{Multiscale TDPSF Propagation Algorithm}
  \label{algo:multiscaleTDPSFPropagator}

  Fix an initial condition $\psi_{0}(x)$, and a number of scales $M$. Set $n=0$.

  \begin{enumerate}
    \item For times $t \in [n \Tstep, (n+1)\Tstep]$, solve $\psi(x,t)$ on the finite domain using operator splitting. The frequency domain operator is approximated by Algorithm \ref{algo:multiscaleSplitStep} (see Section \ref{sec:multiscaleCalculationDO}).
    \item When $t=n \Tstep$, apply phase space filters,replacing $\psi(x,t)$ by:
      \begin{equation*}
        \left[1-\sum_{n=0}^{M}P_{n}^{\textrm{OUT}} \right]\psi(x,n \Tstep)
      \end{equation*}
      The projections are calculated using Algorithm \ref{algo:phaseSpaceProjection} (see Section \ref{sec:phaseSpaceFilters}).
    \item If $n \Tstep > \Tmax$ then stop, otherwise goto step 1.
  \end{enumerate}
  The computational complexity of this algorithm is
  \begin{equation}
    \label{eq:complexityOfWholeAlgorithm}
    O(N M \Tmax \log N) = O(\Tmax \kmax M \log \kmax).
  \end{equation}
  The parameter $N$ is defined as $N = 2L/\delta x=2L \kmax/\pi$.

  Step 2 has this complexity, since we compute $M$ phase space projection operators, each of which has complexity $O(N/4 \log N)$ (see Algorithm \ref{algo:phaseSpaceProjection}, Section \ref{sec:multiscaleTDPSF}).

  It is shown in the Section \ref{sec:multiscaleCalculationDO} that Step 1 has complexity $O(N M \Tstep \log N)$. Since the loop runs no more than $\Tmax/\Tstep$ times, we obtain the correct asymptotic complexity.
\end{algo}

\subsection{Connections to other fields}

It is worth mentioning two important areas that our work relates to. First, interior solvers similar to ours have been constructed to solve the heat equation \cite{MR1766718} and the equation $\sqrt{-\Delta}u=f$ \cite{MR1856303}. The major difference between these works and ours is that they use the natural decay of the differential operator (for large $k$) to suppress high frequencies, while we use filtering. Those works also use the non-uniform FFT, while we make do with the standard one.

We also remark that similar ideas are used in Coarse Grained Molecular Dynamics (CGMD) \cite{rudd:144104,weinane:multiscaleCrystals,weinan:atomisticContinuum}. In CGMD, one wishes to study the dynamics of individual atoms in a small region of space, which are coupled to large scale effects (well modelled by continuum equations) elsewhere. The general ideas are similar to ours, although of course our model is continuous at all levels.

\subsection{Disclaimer}

\label{sec:disclaimer}
The algorithm presented here is intended to be a proof of principle, and is not intended to be directly competitive with other schemes except when $\kmax/\kmin$ is large. Additionally, the proof of accuracy of the differential propagator used makes an important assumption that is not strictly true for any real implementation. We assume that the only errors made by the FFT are caused by aliasing in $x$ and $k$. We neglect errors caused by floating point arithmetic and errors due to discretizing the integral involved in computing the Fourier transform. Since aliasing errors are the dominant part of the error in practice, we believe this assumption is reasonable.

\section{Multiscale Calculation of Differential Operators}
\label{sec:multiscaleCalculationDO}
In this section, we describe a multiscale spectral propagation algorithm (henceforth abbreviated MSP). The MSP uses a non-equispaced grid to a allow the calculation of the Fourier transform of functions which are appropriately localized in phase space. The crucial fact is that the potential interacts with waves only on the phase space regions $\abs{k} \leq C/\abs{x}$.

The basic idea of the algorithm is the following. First suppose we want to apply $S(i \nabla)$ to a function $f(x)$, and suppose $f(x)$ has high frequencies located only in the region $[-L/2,L/2]$. Suppose we compute the Fourier transform of $\chi(x) f(x)$, with $\chi(x)$ a smooth partition of unity that is equal to $1$ on $[-L/2,L/2]$. This Fourier Transform is equivalent to Fourier transforming $f(x)$ for high frequencies. If we then subtract these high frequencies from $f(x)$, the result has only low frequencies. But since the new function has low frequencies only, we can Fourier transform it with a larger sample spacing (and reduced computational cost) with no loss of accuracy.

The MSP introduced below uses this by first localizing high frequencies on small regions of space, and then by using ordinary spectral methods to propagate them. An important technical point is that spectral methods are applied to different regions of space with different lattice spacing, and interpolation must be used to recover the low frequencies on regions of space with fine sampling.

\subsection{Multiscale Approximation of Differential Operators}

We now describe an algorithm for approximating a differential operator $S(i \nabla)$ applied to smooth functions suitably localized in phase space. In particular, we assume that for large $x$, $f(x)$ has no high frequencies remaining.

To begin, define the box $B_{0} = [-L,L]$, with $L$ chosen sufficiently large. Let $\kmax$ be a sufficiently large frequency. Both $\kmax$ and $L$ must be large enough to satisfy \eqref{eq:ConstraintsOnSigmaKmaxL}. In what follows, $\delta_{1}$ is a small parameter, to be adjusted later (see Theorem \ref{thm:accuracyOf:algo:multiscaleDifferentialOperatorCalculation}).

\begin{subequations}
  \label{eq:defOfChiPSigma}
  \begin{multline}
    \label{eq:defOfChi0}
    \chi_{0}(x)  =  \frac{2} {\sqrt{\pi}\sigma} e^{-x^{2}/\sigma^{2}} \star 1_{[-3L/4,3L/4]}(x) \\
    =  \frac{1}{2} \left[ \erf[\sigma^{-1}(x+3L/4)] -  \erf[\sigma^{-1}(x - 3L/4)] \right]
  \end{multline}
  \begin{equation}
    P_{0}(k) = \frac{1}{2} \left[\erf(\sigma(k+3\kmax/8)) - \erf(\sigma(k-3\kmax/8))\right]
  \end{equation}
  \begin{equation}
    \label{eq:ConstraintsOnSigmaKmaxL}
    \frac{8 \erfc^{-1}(\delta_{1})}{\kmax} \leq \sigma \leq \frac{L}{8 \erfc^{-1}(\delta_{1})}
  \end{equation}
\end{subequations}

Here, $\star$ represents convolution. The standard deviation $\sigma$ is chosen so that $\chi_{0}(x \nin [-5L/6,5L/6]) \leq \delta_{1}$ and $\chi_{0}(-4L/6 \leq x \leq 4L/6) \geq 1-\delta_{1}$. This means that $\chi_{0}(x)$ approximates a partition of unity (supported on $[-L/2,L/2]$), with error $\delta_{1}$. Similarly, $P_{0}(k)$ approximates a partition of unity, equal to $1$ on $[-\kmax/4,\kmax/4]$ and smaller than $\delta_{1}$ on $[-\kmax/2,\kmax/2]^{C}$.

\begin{remark}
  The function $\erfc(x)$ has the following bounds \cite[page 297-298]{abramowitz:handbookmathfunctions} for $x > 0$:
  \begin{equation}
    \label{eq:erfcLargeXAsymptotics}
    \frac{2}{\sqrt{\pi}} \frac{e^{-x^{2}}}{x+\sqrt{x^{2}+2}}
    \leq \erfc(x) \leq
    \frac{2}{\sqrt{\pi}} \frac{e^{-x^{2}}}{x+\sqrt{x^{2}+4/\pi}}
  \end{equation}
  This implies that $\erfc^{-1}(\delta)$ has the bound (for $\delta \leq \erfc(1) \leq 0.158)$:
  \begin{equation}
    \label{eq:erfcBound}
    \erfc^{-1}(\delta) \leq \sqrt{\ln(\delta^{-1})+\ln(\pi)/2-\ln(2)} = O(\sqrt{\ln(\delta^{-1})})
  \end{equation}
\end{remark}

  The functions $\chi_{0}(x)$ and $P_{0}(k)$ are the base projections; we now define the scaled versions of them:
  \begin{subequations}
    \label{eq:scalingRelationsChiAndP}
    \begin{eqnarray}
      \chi_{m}(x) & = & \chi_{0}(2^{-m}x) \label{eq:scalingRelationsChi} \\
      P_{m}(k) & = & P_{0}(2^{m}k) \label{eq:scalingRelationsP}
    \end{eqnarray}
  \end{subequations}

  Before continuing, we formalize our assumption on the phase space localization of $f(x)$.
\begin{assumption}
  \label{ass:highFrequenciesLocalized}
  For all $m$, $f(x)$ satisfies:
  \begin{equation}
    \norm{ (1-P_{m}(k))f(x) - \chi_{m}(x)(1-P_{m}(k)) \chi_{m}(x) f(x)}{L^{2}} \leq \delta_{1} \norm{f(x)}{L^{2}}
  \end{equation}
  One should interpret the operator $\chi_{m}(x)(1-P_{m}(k)) \chi_{m}(x)$ as a ``projection'' onto the phase space region $[-2^{m}L,2^{m}L] \times [-2^{-m}\kmax/2,2^{-m}\kmax/2]^{C}$. The boundaries are of course fuzzy, due to the vagaries of phase space localization.

  Taking the union of these sets for $m=0\ldots \infty$, we find that they are a covering (using rectangular boxes) of the hyperbola $\abs{k} \leq C/\jap{x}$.
\end{assumption}

We sample the region $B_{0}$ with spacing $\delta x = 2\pi/3\kmax$. This is sufficient to resolve frequencies no larger than $3\kmax/2$ (we explain shortly the reason for the extra spacing). We will then sample the regions $B_{m+1} \setminus B_{m}$ at the rate $2^{m+1} \delta x$. This is sufficient to resolve the frequencies $2^{-m-1} \kmax$ which may exist in the region $B_{m+1}$, with aliasing errors at most $\epsilon$.

We add an additional assumption on the function $f(x)$.

\begin{assumption}
  \label{ass:differentialOperatorDoesntMoveMuch}
  For all $m$, $S(i \nabla) f(x)$ satisfies:
  \begin{equation}
    \norm{ S(i \nabla) \chi_{m}(x)(1-P_{m}(k))\chi_{m}(x) f(x) }{L^{2}(B_{m}^{C},dx)} \leq \delta_{2} \norm{f(x)}{L^{2}}
  \end{equation}
  That is to say, $S(i \nabla)$ can not spread waves out too far.
\end{assumption}
This assumption is satisfied by $S(i \nabla) = e^{i \Lap \delta t}$ for $\delta t$ sufficiently small, provided the problem does not have frequencies which are too large. This assumption is also satisfied even for singular propagators such as $S(i \nabla) = e^{i \abs{k} \delta t}$.

\begin{remark}
  Results very similar to Assumption \ref{ass:differentialOperatorDoesntMoveMuch} are proved in detail in \cite{us:TDPSFrigorous,stucchio:Thesis:2008} for $S(i \nabla)=e^{i \Lap \delta t}$ (in particular, see Section 6). The general flavor of the result is as follows. Suppose that the mass of $\hat{\psi}(k,t)$ outside the region $[-\kmax,\kmax]$ is bounded by $\epsilon$. Then for a filter of width $w$, if $\delta t < Cw \ln(\delta_{1})/\kmax$, then $\delta_{2}$ is bounded by:
  \begin{equation}
    \delta_{2} \leq C \sqrt{L \kmax} \delta_{1} + C' \epsilon
  \end{equation}
  In fact, the result proved in \cite{us:TDPSFrigorous,stucchio:Thesis:2008} is more general. That result applies to $N$ dimensions (replacing $\sqrt{L \kmax}$ by $(L \kmax)^{d/2}$), as well as non-differential propagators: instead of $e^{i \Lap \delta t}$, one can use $e^{-i H \delta t}$ with $H=-\Lap+V(x)$ for $V(x)$ supported (mostly) inside $[-L/2,L/2]$.
\end{remark}

We make one further assumption.
\begin{assumption}
  \label{ass:unitary}
  We assume that $\abs{S(k)} = 1$.
\end{assumption}
This assumption is not strictly necessary, but it makes the proof simpler. Since we are interested in wave propagators of the form $S(k) = e^{i \omega(k) t}$, this is a perfectly reasonable restriction.

We are now prepared to discuss the algorithm. When discussing the computational complexity, let $N=2L/\delta x=2L \kmax/\pi$.

\begin{definition}
  The operator $\FFT{\delta x}$ is the $N$-point Fast Fourier transform algorithm operating on the points $\{-(N/2) \delta x, -(N/2+1) \delta x,\ldots, (N/2) \delta x \}$. The computational cost is $O(N \log N)$. The inverse is denoted by $\FFTi{\delta x}$.
\end{definition}

\begin{algo}
  \dueto{Phase Space Localization}
\label{algo:multiscaleFFT}
  This algorithm decomposes a function $f(x)$ into components which are well localized in phase space.

  \begin{enumerate}[1]
  \item
  For $m=0$, define
  \begin{subequations}
    \begin{equation}
      f_{0}^{+}(x)=  \chi_{0}(x) \FFTi{\delta x} (1-P_{0}(k)) \FFT{\delta x}\chi_{0}(x)f(x)
    \end{equation}
    \begin{equation}
      f_{0}^{-}(x) = f(x) - f_{0}^{+}(x)
    \end{equation}
  \end{subequations}
  The cost of using the FFT region $B_{0}$ is $O(N \log N)$, while multiplication is $O(N)$. No errors are caused by using the FFT with periodic boundaries due to the fact that $(1-\chi_0(x))$ vanishes near the boundary of $B_{0}$.

  Additionally, define $\hat{f}_{0}^{+}(k)=0$ for $\abs{k} > \kmax$. This is done to control discretization errors.
  \item
    For $m \geq 1$, define:
    \begin{subequations}
      \label{eq:defOfFplusFminus}
      \begin{equation}
        f_{m}^{+}(x)= \chi_{m}(x)\FFTi{2^{m}\delta x} (1-P_{m}(k)) \FFT{2^{m}\delta x} \chi_{m}(x) f^{-}_{m-1}(x)
      \end{equation}
      \begin{equation}
        f_{m}^{-}(x) = f_{m-1}^{-}(x) - f_{m}^{+}(x)
      \end{equation}
    \end{subequations}
    While computing the operator $(1-P_{m}(k))$, we set $(1-P_{m}(k))=0$ for $\abs{k} > (2/3) 2^{-m} \kmax$ to control discretization errors.
    \item
      Return the list of functions $[ f^{+}_{0}(x), f^{+}_{1}(x), \ldots, f^{+}_{M}(x), f^{-}_{M}(x)]$.
  \end{enumerate}

  The computational cost of this algorithm is $O(M N \log N)$. This follows because for $m=0\ldots M$, we compute an FFT of a region with $N$ points at a cost $O(N \log N)$. In addition, we must compute an FFT of a region with $N/2$ points at a cost $O(N/2 \log N)=O(N \log N)$.
\end{algo}

\begin{remark}
  Algorithm \ref{algo:multiscaleFFT} will NOT work for an arbitrary function $f(x)$. This works for $f(x)$ satisfying Assumption \ref{ass:highFrequenciesLocalized} simply because the downsampling used on the coarse grids cannot alias high frequencies which are not present.
\end{remark}

\begin{remark}
  In steps 1 and 2 of Algorithm \ref{algo:multiscaleFFT}, we set high frequencies ($\abs{k} \geq 2^{-m}\kmax$ on each scale) equal to zero. On the level of infinite precision arithmetic, this appears unnecessary. In a real computer, floating point errors occur which can cause long time instability of the numerical solution. Oversampling the grid by a factor of $3/2$ and filtering the high frequencies appears to solve this problem. This works because floating point errors can be assumed to have arbitrary frequencies, and we therefore remove $(1/3)$ of them per timestep, whereas the effect on the true solution is negligible.
\end{remark}

We have the following result as far as the accuracy of Algorithm \ref{algo:multiscaleFFT}, which is proved in Appendix \ref{sec:proveCorrectAlgo:MultiscaleFFT}.

\begin{theorem}
  \label{thm:accuracyOfMultiscaleFFT}
  Suppose that $f(x)$ satisfies Assumption \ref{ass:highFrequenciesLocalized}. Let $f^{\pm,d}_{m}$ be the discretization of $f^{\pm}_{m}$, computed according to Algorithm \ref{algo:multiscaleFFT} with $M$ scales. Suppose further that $\sigma, \kmax$ and $L$ satisfy \eqref{eq:ConstraintsOnSigmaKmaxL}. Then:
  \begin{subequations}
    \begin{eqnarray}
      \norm{f^{+}_{m}(x) - f^{+,d}_{m}(x)}{L^{2}} & \leq & \delta_{1}\left(2m +\frac{16 \kmax }{\sqrt{2\pi} \sigma} \right) \norm{f(x)}{L^{2}}\\
      \norm{f^{-}_{M}(x)-f^{-,d}_{M}(x)}{L^{2}} &\leq &    \delta_{1} \left(2M + \frac{16 \kmax }{\sqrt{2\pi} \sigma} \right)\norm{f(x)}{L^{2}}
    \end{eqnarray}
  \end{subequations}
\end{theorem}

This result shows the error is linear in the number of scales, and linear in $\kmax$. We conjecture that the dependence on $\kmax$ could be removed; see Remark \ref{rem:kmaxDueToLaziness}.

Before we continue, we provide an interpolation method which provides spectral accuracy.

\begin{algo}
  \label{algo:spectralInterpolation}
  \dueto{Spectral Interpolation}
  Let $f(x)$ be a function on $B_{m}$ with lattice spacing $2^{m} \delta x$. $f(x)$ must also be localized in frequency on the region $[-2^{-m} \kmax,2^{-m}\kmax]$.

  Then $\SI{m} f(x)$ is a spectral approximation to $f(x)$ on $B_{m-1}$ with lattice spacing $2^{m-1}\delta x$. Calculation of $\SI{m} f(x)$ is done as follows.
  \begin{enumerate}[1]
  \item Compute the Fast Fourier Transform of $\chi_{m-1}(x) f(x)$ on $B_{m}$ with lattice spacing $2^{m}\delta x$.
  \item Compute the inverse Fast Fourier Transform on $B_{m}$ with lattice spacing $2^{m-1}\delta x$. Frequencies with $\abs{k} \geq 2^{-m} \kmax$ are populated with zeros before computing this. Discard lattice points outside $B_{m-1}$. The result is $\SI{m} f(x)$.
  \end{enumerate}
  Since this algorithm is merely an FFT and an inverse FFT (on grids with $N$ and $2N$ lattice points, respectively), the computational cost is $O(3N \log N) = O(N \log N)$.
\end{algo}

We now describe the propagation algorithm.

\begin{algo}
  \dueto{Multiscale Spectral Propagator}
  \label{algo:multiscaleDifferentialOperatorCalculation}
  \begin{enumerate}[1]
  \item Compute $[ f^{+}_{0}(x), f^{+}_{1}(x), \ldots, f^{+}_{M}(x), f^{-}_{M}(x)]$ by means of Algorithm \ref{algo:multiscaleFFT}.
  \item For each $m$, compute $S(k)\FFT{2^{m}\delta x} f^{+}_{m}$ as well as $S(k) \FFT{2^{M}\delta x} f^{-}_{M}$. This has computational cost $O(M N \log N)$.
  \item Compute the functions:
    \begin{subequations}
      \begin{equation}
        g_{M}(x) = \FFTi{2^{M}\delta x} S(k)\FFT{2^{M}\delta x} f^{+}_{M} + \FFTi{2^{M}\delta x} S(k) \FFT{2^{M}\delta x} f^{-}_{M}
      \end{equation}
      \begin{equation}
        g_{m-1}(x) = \FFTi{2^{m-1}\delta x} S(k)\FFT{2^{m-1}\delta x} f^{+}_{m-1} + \SI{m-1} g_{m}(x)
      \end{equation}
    \end{subequations}
    The function $g_{m}(x)$ (defined on $B_{m}$) approximates
    \begin{equation*}
      S(k) \left[f^{-}_{M} + \sum_{k=m}^{M}f^{+}_{k} \right] \approx S(k)P_{m}f
    \end{equation*}
    on the region $B_{m}$ with lattice spacing $2^{m}\delta x$.
  \item Return the function $g(x)$ given by $g(x)=g_{m}(x)$ for $x \in B_{m} \setminus B_{m-1}$.
  \end{enumerate}
\end{algo}

We have the following result as to the accuracy of Algorithm \ref{algo:multiscaleDifferentialOperatorCalculation}.

\begin{theorem}
  \label{thm:accuracyOf:algo:multiscaleDifferentialOperatorCalculation}
  Suppose $S(k)$ satisfies Assumptions \ref{ass:differentialOperatorDoesntMoveMuch} and \ref{ass:unitary}, and $f(x)$ satisfies Assumption \ref{ass:highFrequenciesLocalized}. Suppose also that \eqref{eq:ConstraintsOnSigmaKmaxL} holds.

  Now let $g(x) = g_{0}(x)$ as calculated by Algorithm \ref{algo:multiscaleDifferentialOperatorCalculation}, but using exact Fourier transforms on $\mathbb{R}$ instead of the FFT on truncated regions. Let $g^{d}(x)$ be the discrete approximation to $g(x)$ (calculated with space and frequency truncation, but with no integration/floating point error).

  Then we have the following bounds:
  \begin{subequations}
    \label{eq:mainErrorBoundFormula}
    \begin{equation}
      \label{eq:multiscaleDifferentialOperatorErrorBound}
      \norm{ S(k) f(x) - g^{d}(x)}{L^{2}} \leq
      \norm{ S(k) f(x) - g(x)}{L^{2}} + \norm{g(x) - g^{d}(x)}{L^{2}}
    \end{equation}
    with:
    \begin{multline}
      \label{eq:multiscaleDifferentialErrorBound:assumptionErrors}
      \norm{ S(k) f(x) - g(x)}{L^{2}} \\
      \leq \delta_{1} \left(
        M^{3}+10M^{2}+(55/2+\kmax)M+225\kmax+18
      \right)\norm{f}{L^{2}}\\
      + 2\norm{P_{M} f}{L^{2}}
    \end{multline}
    and
    \begin{multline}
      \label{eq:multiscaleDifferentialErrorBound:discretizationErrors}
      \norm{ g(x) - g^{d}(x)}{L^{2}} \leq
        \delta_{1} \left( 5M^{2}+\left[  15+\frac{16 \kmax }{\sqrt{2\pi} \sigma} \right]M
        + 12+150\kmax \right)\\
      + \delta_{2} M \norm{f(x)}{L^{2}} + 2\norm{P_{M} f}{L^{2}}
    \end{multline}
  \end{subequations}
  These two formulas should be interpreted as follows.

  Equation \eqref{eq:multiscaleDifferentialErrorBound:assumptionErrors} is the error caused discarding those parts of the solution which Assumption \ref{ass:highFrequenciesLocalized} says are small. On the coarsest scale, we discard waves which might be significant even if Assumption \ref{ass:highFrequenciesLocalized} is true, causing the presence of the last term.

  Equation \eqref{eq:multiscaleDifferentialErrorBound:discretizationErrors} is the error caused by approximating the Fourier transform by the discrete FFT algorithm. This is caused both by the fact that our localization procedure introduces tails in $x$ and $k$ (the first term in \eqref{eq:multiscaleDifferentialErrorBound:discretizationErrors}) and the fact that  $S(k)$ spreads waves out beyond the limits of truncation (the second term). The last term is caused by spatial truncation errors in the lowest frequencies.
\end{theorem}

We now present the propagation Algorithm for equation \eqref{eq:schroIVP}.

\begin{algo}
  \dueto{Multiscale Schrodinger Solver}
  \label{algo:multiscaleSchrodingerSolver}
  Take input $\psi_{0}(x)$, and fix $\delta t > 0$. This algorithm approximates $\psi(x,n \delta t)$ for $n$ an integer.
  \label{algo:multiscaleSplitStep}
  \begin{enumerate}
  \item Iterate over $n=0..n_{max}=\Tmax/\delta t$.
    \begin{enumerate}[i]
    \item Define $\psi_{n,i}(x)= e^{i (\delta t/2) \Lap} \psi_{n}(x)$. Calculate this using Algorithm \ref{algo:multiscaleDifferentialOperatorCalculation}.
    \item Define $\psi_{n,ii}(x)= e^{i \delta t V(x)} \psi_{n,i}(x)$.
    \item Compute $\psi_{n+1}(x)=e^{i (\delta t/2) \Lap} \psi_{n,ii}(x)$, again using Algorithm \ref{algo:multiscaleDifferentialOperatorCalculation}.
    \end{enumerate}
  \end{enumerate}
  This algorithm has complexity $O(n_{max} N M \log N)$. This can be seen since the application of Algorithm \ref{algo:multiscaleDifferentialOperatorCalculation} in steps (i) and (iii) have complexity $O(N M \log N)$ while step (ii) has complexity $O(N M)$.

  This algorithm provides $O(\delta t^{3})$ accuracy ($O(\delta t^{2})$ for time-dependent potentials).
\end{algo}

This algorithm is merely the standard 2'nd order operator splitting method. The only difference is that we now use the multiscale propagator to approximate the $e^{i (\delta t/2) \Lap}$ step instead of the usual FFT.

\subsection{Interpreting the error bound}
\label{sec:interpretingTheError}

The standard spectral propagator, namely $\FFTi{\delta x} e^{i k^{2} t} \FFT{\delta x}$ is said to be spectrally convergent for the following reason. Assume $f(x)$ decays exponentially in $x$ and $\hat{f}(k)$ decays exponentially in $k$. Assume that $e^{i \Delta t} f(x)$ is localized within $[-L,L]$ up to an error  of size $\delta_{2} \norm{f}{L^{2}}$, where $[-L,L]$ is the computational box. Then the spectral propagator can be applied to compute $e^{i \Delta t}f(x)$ with error proportional to $\delta_{1}+\delta_{2}$ at computational cost $O(N \log N$), with $N=O(\log (\delta_{1}^{-1}))$.

Let us discuss why Algorithm \ref{algo:multiscaleDifferentialOperatorCalculation} provides this level of accuracy.

First, observe that the computational complexity is $O(M N \log N)$, with $N=L \kmax$. Rearranging \eqref{eq:ConstraintsOnSigmaKmaxL}, we obtain the constraint $64 (\erfc^{-1}(\delta_{1}))^{2} \leq L \kmax$. Examining \eqref{eq:erfcBound} shows that $\erfc^{-1}(\delta_{1}) = O(\sqrt{\log (1/\delta_{1})} )$ for small $\delta_{1}$, which implies that:
\begin{equation*}
  N = L \kmax = O(\log (\delta_{1}^{-1}))
\end{equation*}
This is spectral convergence in $N=L \kmax$; i.e. the error caused by discretizing the problem decays exponentially as computational cost increases, provided of course that the function is localized sufficiently well in phase space (though not exponentially decaying in $x$ and $k$, as before).

There are two additional parameters present: $M$ and $\norm{P_{M} f}{L^{2}}$. As $M$ increases, we must decrease $\delta_{1}$ correspondingly, causing only a logarithmic increase in $N$. $\delta_{2}$ depends strongly on $S(k)$, and is analogous to errors caused by aliasing (waves leaving the right side of the box and entering the left).

The last term is a smooth projection onto waves with frequency $k \leq 2^{-M} \kmax$. This can be bounded if we assume some smoothness of $\hat{f}(k)$:
\begin{multline}
  \label{eq:coarseScaleErrorPreVirial}
  \norm{P_{M} f}{L^{2}}^{2} \leq \delta_{1} \int_{\abs{k} > 2^{-M}\kmax/2} \abs{ \hat{f}(k)}^{2} dk + \int_{-2^{-M} \kmax/2}^{2^{-M} \kmax/2} \abs{ \hat{f}(k)}^{2} dk  \\
  \leq \delta_{1}^{2} \norm{f}{L^{2}}^{2} + 2^{-M} \kmax \norm{\hat{f}(k)}{L^{\infty}}^{2}  \\
  \leq \delta_{1}^{2} \norm{f}{L^{2}}^{2} + 2^{-M} \kmax C \norm{\jap{x}f(x)}{L^{2}}^{2}
\end{multline}
The constant $C$ comes from the Sobolev embedding theorem\footnote{Note that more work is required in higher dimensions, since $H^{1}(\mathbb{R}^{2,3}) \not\hookrightarrow L^{\infty}(\mathbb{R}^{2,3})$.}, since $  \norm{\hat{f}(k)}{L^{\infty}} \leq \sqrt{C} \norm{\hat{f}(k)}{H^{1}} = \sqrt{C}\norm{\jap{x}f(x)}{L^{2}}$.

For Schr\"odinger equations, we expect that $\norm{\jap{x}f(x)}{L^{2}} \sim \jap{p} t$ (see Section \ref{sec:virialTheoremBoundsLowMomenta} below). This means that the error due to low velocities behaves like $O(2^{-M/2} \kmax^{1/2})$, which yields exponential decay in $M$. Therefore, we have spectral convergence in $N$ and $M$; linear increases in $N=L \kmax$ decrease $\delta_{1}$ exponentially, while linear increases in $M$ decrease the error due to low frequencies exponentially. The remaining error is due to the wave propagator moving waves outside the phase space region of interest, which is present in any spectral method.

\subsubsection{Why do we expect $\norm{\jap{x}f(x)}{L^{2}}^{2}$ to be bounded?}
\label{sec:virialTheoremBoundsLowMomenta}

When used to propagate $e^{i \Delta t}$ as in Algorithm \ref{algo:multiscaleSplitStep}, $f(x)$ will be an approximation to $\psi(x, n \delta t)$ with $\psi(x,t)$ solving the Schr\"odinger equation. For the Schr\"odinger equation, the virial theorem says that:
\begin{equation}
  \label{eq:23}
  \norm{\jap{x} \psi(x,t)}{L^{2}}^{2} \leq \norm{\jap{x} \psi(x,0)}{L^{2}}^{2} + t^{2} \norm{ \psi(x,t)}{H^{1}}^{2}
\end{equation}
In physical terms, this is merely the statement that $\langle x  \rangle \leq \langle x_{0}  \rangle + \langle p  \rangle t$, with $\jap{p}$ the expected value of momentum. For other types of wave equations the precise definition of momentum needs to be changed, but similar identities usually hold.

Moreover, provided $\delta_{2}$ is small, this provides us with a general idea of how long a numerical simulation (using Algorithm \ref{algo:multiscaleTDPSFPropagator}) will remain valid. Using \eqref{eq:23}, we see that provided $2^{-M} \sqrt{\kmax}(\langle x_{0} \rangle+\langle p \rangle t)$ is small, the error will remain controlled. This implies that if error $\epsilon$ is desired, using
\begin{equation*}
  M = O(\log \epsilon^{-1} + \log (\langle x_{0} \rangle + \langle p \rangle t))
\end{equation*}
 scales will control the error.

In practice, this appears to be an overestimate. Moreover, it suggests that as we increase $M$, we can use Algorithm \ref{algo:multiscaleTDPSFPropagator} for $0 \leq t \leq \Tmax = O(2^{M/2})$. This behavior is observed in numerical experiments; see Section \ref{sec:longRangeTests} and Figure \ref{fig:multiscalePotentialError}.

\subsection{Possible Improvement}

\label{sec:fftAlgorithmPossibleImprovement}
Although algorithms \ref{algo:multiscaleFFT} and \ref{algo:multiscaleDifferentialOperatorCalculation} are more efficient than FFT-based spectral methods for large $M$, they fail to be competitive for small $M$ (in particular $M < 5$). Careful optimization may reduce this, but made no attempt to do this since our main concern is efficiency as $M$ becomes large.

A more serious problem is that while dyadic downsampling works for $V(x) \sim C x^{-2}$ (near $x=\infty$), it will not work for $V(x) \sim -Z/\abs{x}$ since the interaction region is $\{ (x,k) : \abs{k} \leq C/\jap{x}^{1/2} \}$.  For sufficiently large $m$, Algorithm \ref{algo:multiscaleTDPSFPropagator} will attempt to filter waves near $x=2^{m}L, k=2^{-m}\kmax$. But the energy of these waves is $(2^{-m} \kmax)^{2} - C/(2^{m}L) < 0$, implying that Algorithm \ref{algo:multiscaleTDPSFPropagator} will attempt to filter waves  with negative energy. This is a serious problem, stemming from the fact that the interaction region is larger in this case. For these problems, we propose using boxes of size $B_{m} = [-2^{2m}L,2^{2m}L]$ each with lattice spacing $2^{m}\delta x$ (implying that there are $O(2^{m})$ lattice points per scale instead of a constant number). Although this is more costly, we believe it will resolve the problem for the case of Coulomb potentials (and is still cheaper than using a rectangular phase space region).

\section{Phase Space Filtering}
\label{sec:phaseSpaceFilters}

\subsection{Introduction to Phase Space Filtering}

In this section we describe the phase space filtering algorithm which will be applied. We will apply the phase space filter on a region having size $L/4$ on each side of the grid. This choice is almost certainly not the most efficient, but it is the simplest. Consider the following algorithm (see \cite{us:TDPSFjcp} for details):
\begin{algo}\dueto{Simplified TDPSF Propagation Algorithm}
  \label{algo:simpleTDPSF}

  Define $w_{L}(x)$ to be zero outside $[-L,-L/2]$, and to be $1$ on some region inside $[-5L/6,-4L/6]$, and $w_{R}(x)=w_{L}(-x)$. Suppose that $\hat{w}_{L,R}(k)$ is localized in the frequency band narrower than $[-\kmax/4,\kmax/4]$ (this is possible provided $L$ is large enough).

  Further, define $\chi_{\pm}(k) = 1$ for $k \in [\pm \kmax/4,\pm \infty)$, $0$ for $k < \kmax/8$ and smooth in between.
  \begin{enumerate}
  \item For $t \in [n \Tstep, (n+1)\Tstep]$, solve \eqref{eq:schroIVP} by using the Split Step Fourier method on $[-L,L]$. We assume the initial condition is localized in the region $[-L/2,L/2]$.
  \item At times $n \Tstep$, replace $\Psi(x,n \Tstep)$ by:
    \begin{equation*}
      [1-w_{L}(x) \chi_{-}(k)w_{L}(x) - w_{R}(x)\chi_{+}(k)w_{R}(x)]\Psi(x,n \Tstep)
    \end{equation*}
    The operator $w_{L}(x) \chi_{-}(k)w_{L}(x)$ is a projection onto rightward moving waves located to the right of $x=L/2$, and the $w_{R}(x) \chi_{+}(k)w_{R}(x)$ is a similar operator.

    The time $\Tstep$ is chosen to be $(L/4)/ \kmax$, to guarantee that waves with velocity $\kmax$ or slower cannot pass through the region $[L/2,3L/4]$ before being filtered.
  \end{enumerate}
\end{algo}

In practice, we simply take $w_{L}(x)=\chi_{[-L+\sqrt{\ln \epsilon},-L/2-\sqrt{\ln \epsilon}]}(x) \star e^{-x^{2}/2}$, with $\star$ denoting convolution. This doesn't quite satisfy the localization condition, but comes close enough for practical purposes if $\epsilon$ is small enough and $L$ is large enough.

This algorithm is capable of approximating the solution to \eqref{eq:schroIVP} on the region $[-L/2,L/2]$, provided $\psi(x,t)$ has no outgoing waves with frequency $k < \kmax/4$. This can be seen for the following reason. A wave with velocity $k \in [\kmax/4,\kmax]$ can travel from $[L/2,3L/4]$ before being removed by the application of the operator $[1-w_{L}(x) \chi_{-}(k)w_{L}(x) - w_{R}(x)\chi_{+}(k)]$. Since the wave never reaches the boundary at $x=L$, the boundary conditions are irrelevant, and periodic boundaries can be used.

This is a simplified version of the Time Dependent Phase Space Filter (TDPSF) constructed in \cite{us:TDPSFrigorous,stucchio:Thesis:2008,us:TDPSFjcp}. The version in \cite{us:TDPSFrigorous,stucchio:Thesis:2008,us:TDPSFjcp} is more efficient in many ways. Most notably the buffer region is taken to have width $w$ (a parameter independent of $L$) rather than merely $L/2$, and can filter waves with frequency as low as $(\ln \epsilon / w)$, with $\epsilon$ the desired error. In addition, the approach of \cite{us:TDPSFrigorous,stucchio:Thesis:2008,us:TDPSFjcp} works in multiple dimensions, while Algorithm \ref{algo:simpleTDPSF} does not.

However, the critical problem with both the approach of \cite{us:TDPSFjcp,us:TDPSFrigorous,stucchio:Thesis:2008} and Algorithm \ref{algo:simpleTDPSF} is that they cannot accurately filter low frequency waves. In fact, this problem is shared by nearly all methods of open boundaries, including absorbing potentials \cite{neuhauser:complexPotentials}, the PML \cite{MR1294924} and Dirichlet-to-Neumann boundaries \cite{MR0471386,MR596431,MR1913093,szeftel:absorbingBoundaries,MR2147901,MR1869342,jian:thesis,MR2060329}. The reason for this varies depending on the method. Dirichlet-to-Neumann boundaries, in all but the simplest (free wave) cases, are constructed by means of high frequency approximations, and therefore fail for low frequencies. Absorbing potentials and the PML have problems with reflection/lack of filtering for small $k$, analogous to errors made by the TDPSF.


\subsection{The Multiscale TDPSF}

\label{sec:multiscaleTDPSF}

Instead of increasing the width of the filter, thereby lowering $\kmin$, we simply downscale the filter so that it covers the edge of the interaction region. That is, we apply the same filtering procedure (with the details clarified), but additionally using filters with $w(2^{-m}x)$ and $\chi_{\pm}(2^{m}k)$ replacing $w(x)$ and $\chi_{\pm}(k)$. With $M$ scales this allows us to filter outgoing waves with $k$ as low as  $O(2^{-M} \kmax)$.

We now construct the Multiscale TDPSF. Let $w_{\pm}(x)$ be a smooth partition of $[\pm L/2,\pm L]$, and let $w_{\pm,n}(x) = w_{\pm}(2^{-n}x)$. Similarly, let $\chi_{\pm,n}(k)=\chi_{\pm}(2^{n}k)$. We define these operators precisely soon. The operator
\begin{equation}
  \label{eq:defOfPn}
  P_{n}^{\textrm{OUT}}=w_{+,n}(x) \chi_{+,n}(k)w_{+,n}(x) +w_{-,n}(x) \chi_{-,n}(k)w_{-,n}(x)
\end{equation}
is then an approximate projection onto outgoing waves. For concreteness, we fix a tolerance $\epsilon$, and define:
\begin{subequations}
  \begin{equation}
    w_{+}(x) = (1/2)[\erf((x-(L-b))/\sqrt{2})+\erf((x-(L/2+b))/\sqrt{2})]
  \end{equation}
  \begin{equation}
    b = \erf^{-1}(\epsilon) = O(\sqrt{\ln(\epsilon^{-1})})
  \end{equation}
  \begin{equation}
    \label{eq:boundOnEpsilonwrtL}
    b < L/12
  \end{equation}
\end{subequations}
with $w_{-}(x)$ defined similarly. The constant $b$ is chosen so that $w_{+}(x) < \epsilon$ to the left of $L/2$ and to the right of $L$. For this choice to be reasonable, we require that $L > 2b$. The function $\chi_{+}(x)$ is given by
\begin{subequations}
  \label{eq:defOfFrequencyCutoff}
  \begin{equation}
    \chi_{+}(k) = \erfc( (k - \kmax/2-b')/2)
  \end{equation}
  \begin{equation}
    b' = (3/2) \erf^{-1}(\epsilon/L) = O(\sqrt{\ln(L/\epsilon)})
  \end{equation}
  \begin{equation}
    \label{eq:boundsOnEpsilonwrtKmax}
    b' < \kmax/4
  \end{equation}
\end{subequations}
This means that $\chi_{+}(k) \geq 1-\epsilon$ when $k > 2b'$. This cutoff is chosen so that even after multiplication by $w_{\pm}(x)$ (which corresponds to convolution in the $k$ domain) $P_{0} f(x)$ will have no more than $\epsilon$ frequency content below $k=0$.

The constraints \eqref{eq:boundOnEpsilonwrtL} and \eqref{eq:boundsOnEpsilonwrtKmax} provide the ultimate limits on the error of the phase space filter method. As $\epsilon$ decreases, $b$ and $b'$ increase, thereby requiring $L$ and $\kmax$ to increase. The number of lattice points required is $N=O(L \kmax) \geq O(b b') \geq O(\ln(\epsilon))$\footnote{Our analysis actually suggests $O(\sqrt{\ln(\epsilon^{-1})^{2}+\ln(L) \ln(\epsilon^{-1})})$, but the distinction only matters when the region of interest is extremely large compared to the accuracy desired, e.g. $L \geq \epsilon^{-1}$. We believe this dependence is an artifact of our analysis: when studying convolutions we take absolute values, bounding $\abs{2 (L/4-2b)\sinc( (L/4-2b)k) e^{-x^{2}/2}}$ by $(L/4-2b) e^{-x^{2}/2}$. This is suboptimal, but good enough for our purposes.}.

The operators $P^{\textrm{OUT}}_{n}$ can be calculated as follows.
\begin{algo}
  \label{algo:phaseSpaceProjection}
  \dueto{Phase Space Projections}

  Take as input $f(x)$, a function defined on an arbitrarily large domain. This algorithm computes $P^{\textrm{OUT}}_{n} f(x)$. Assume that $N$ lattice points are used to sample the region $[-L,L]$, and that the sampling rate decreases dyadically on later regions.
  \begin{enumerate}
  \item Truncate the domain to $[-2^{n}L, -2^{n}L/2] \cup [2^{n}L/2, 2^{n}L]$.
  \item Compute $w_{+,n}(x) f(x)$ on $[-2^{n}L, -2^{n}L/2]$ only, and $w_{-,n}(x) f(x)$ on\\ $[2^{n}L/2, 2^{n}L]$ only.
  \item Take the FFT of $w_{\pm,n}(x) f(x)$, and multiply by $\chi_{\pm,n}(k)$. Note that the FFT of each region is done separately.
  \item Invert the FFT and multiply the result by $w_{\pm,n}(x)$. This yields $[P^{\textrm{OUT}}_{n}f](x)$ for $x \in [-2^{n}L, -2^{n}L/2] \cup [2^{n}L, 2^{n}L/2]$.
  \item For all other $x$, approximate $P^{\textrm{OUT}}_{n}f(x)$ by $0$.
  \end{enumerate}
  The computational complexity of this algorithm is $O(N \log N)$, which can be seen simply by noting that the FFT has this complexity. All other steps are $O(N)$ or $O(1)$.
\end{algo}

The number of lattice points in Algorithm \ref{algo:phaseSpaceProjection} is only $N/4$ because the region $[2^{n}L/2,2^{n}L]$ is sampled with sample spacing $2^{n}\delta x$ rather than simply $\delta x$.

We are implicitly assuming that the maximal frequency found in $[2^{n}L/2,2^{n}L]$ is $2^{-n}\kmax$ rather than $\kmax$. This assumption actually holds because waves with frequency larger than $2^{-n}\kmax$ are filtered before they can reach the $n$'th layer.

\subsection{Accuracy of the Filtering Algorithm}

\label{sec:TDPSFErrors}

In \cite{us:TDPSFrigorous,stucchio:Thesis:2008}, it is proven rigorously that if $V(x) \approx 0$ in the filter region, Algorithm \ref{algo:simpleTDPSF} is accurate\footnote{We actually prove the result for a variant of Algorithm \ref{algo:simpleTDPSF} which extends to multiple dimensions.}. We show that if the solution $\psi(x,t)$ has frequencies bounded between $K$ and $\kmax$ and if $V(x)$ is well localized, a phase space filter of width $O(\ln(\epsilon)/K)$ will provide accuracy of order $O(\Tmax \epsilon)$ (ignoring logarithmic prefactors). In \cite{us:TDPSFjcp,us:TDPSFrigorous,stucchio:Thesis:2008}, $K$ is taken to be $\kmin$, while in this work we take $K=\kmax/2$ (leaving low frequencies to be filtered by the downscaled filters). These error bounds are proven only for rapidly decaying $V(x)$. Proving accuracy for slowly decaying $V(x)$ is more difficult, and involves showing that $P^{\textrm{OUT}}_{n}$ projects accurately onto outgoing waves of $-\Lap+V(x)$ (a nontrivial problem of scattering theory).

To demonstrate the accuracy of a single phase space filter, we solved \eqref{eq:schroIVP} using Algorithm \ref{algo:simpleTDPSF}. The initial condition was $\psi_{0}(x)=e^{i k x} e^{-x^{2}/2\cdot49}$. The computational region was $x \in [-51.2,51.2]$ and $k \in [-31.4,31.4]$ ($1024$ lattice points), and $\Tmax$ was taken to be $4L/k$ (enough time for waves to wrap around the box twice). A phase space filter with varying $\epsilon$, but fixed $L$ and $\kmax$, was used to remove outgoing waves. Waves with frequency $k \geq 13$ are accurately filtered, while low frequency waves are not. The results are plotted in Figure \ref{fig:PhaseSpaceFilterError}.

\begin{figure}
\setlength{\unitlength}{0.240900pt}
\ifx\plotpoint\undefined\newsavebox{\plotpoint}\fi
\sbox{\plotpoint}{\rule[-0.200pt]{0.400pt}{0.400pt}}%
\includegraphics[scale=0.55]{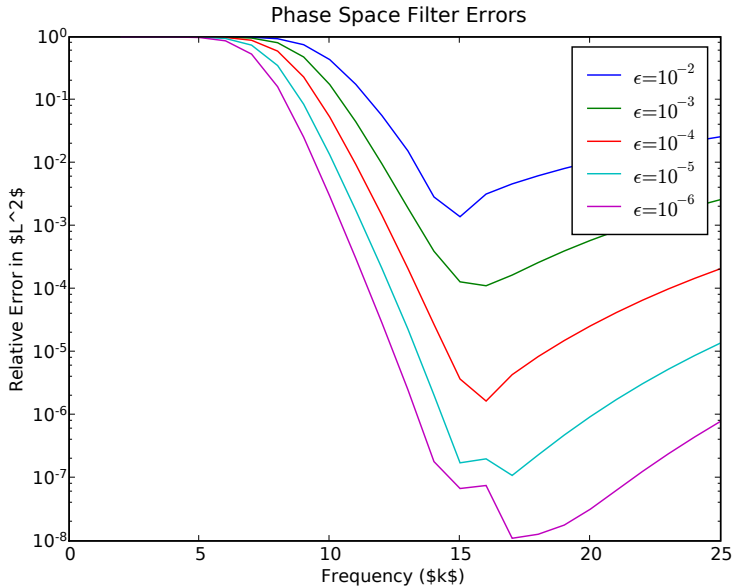}
\caption{A graph of error ($\norm{\psi(x,t)-\psi_{d}(x,t)}{L^{2}([-L,L])} / \norm{\psi_{0}(x)}{L^{2}}$) vs frequency of an outgoing pulse for the phase space filtering algorithm.
}
\label{fig:PhaseSpaceFilterError}
\end{figure}

\section{Justification of the Method}

We discuss briefly why Algorithm \ref{algo:multiscaleTDPSFPropagator} works.

Begin by assuming that $\psi(x,0)=\psi_{L}(x) + \psi_{B,+}(x)+\psi_{B,-}(x)$, where $\psi_{L}(x)$ is localized on $[-L/2,L/2]$, $\psi_{B,+}(x)$ is localized on $[L/2,L]$ and has frequencies primarily $k \geq \kmax/2$, and $\psi_{B,-}(x)$ is localized on $[L/2,L]$ and has frequencies primarily on $k \leq \kmax/2$. (For simplicity, we consider waves on the right side only, the left being treated identically.)

The component $\psi_{L}(x)$ can be propagated safely for at least a time $\Tstep=L/4 \kmax$. This follows since the maximal velocity is $\kmax$, and it would take a time $2 \Tstep$ for $\psi_{L}(x)$ to cross the buffer region.

The function $\psi_{B,-}(x)$ can be propagated safely as well, at least for time $L/(\kmax/2) = 8 \Tstep$. This is the time it takes for waves with velocity $\kmax/2$ to cross a region of width $L$. The region in question is the low frequency buffer that is added to the edge of the grid (the second scale).

Finally, $\psi_{B,+}(x)$ can not be propagated safely, since it consists of high frequency waves which are about to enter the second scale; however, the first filter removes $\psi_{B,+}(x)$ before it enters the second scale and corrupts the solution.

The same argument can be repeated for arbitrarily many scales, until $2^{-m} \kmax < \kmin$. Thus we see that if $\kmin$ is known, we can solve the problem using $m = O(\log(\kmax/\kmin))$ scales.

\subsection{What if $\kmin$ is unknown?}

\label{sec:unknownKmin}

In the event that $\kmin$ is unknown, we can still solve the problem if $\Tmax$ is known and finite. Suppose we have $M$ scales. Then outgoing waves with frequencies in the range $[2^{-M} \kmax, \kmax]$ can be successfully filtered. The only waves which can cause a problem are waves with $k \in [-2^{-M}\kmax,2^{-M}\kmax]$. Supposing the initial condition to be localized on $[-L/2,L/2]$, these waves can travel no farther in space than $L/2+2^{-M}\kmax \Tmax$. Due to their low frequency, these waves are located on the coarsest scale, which extends as far as $x = 2^{M}L$.

This implies that the slow waves hit the boundary when $2^{-M} \kmax \Tmax \approx 2^{m} L$. We therefore choose $m$ so that
\begin{equation}
  \label{eq:maximalSimulationTime}
  \Tmax \ll \frac{ 2^{2M} L}{\kmax}
\end{equation}
so that the low frequency waves do not have sufficient time to reach the boundary. Thus, with an unknown $\kmin$ but known (large) $\Tmax$, the choice $M \gg \log ( \Tmax \kmax L^{-1})$ guarantees the accuracy of the simulation. The complexity of the algorithm is therefore
\begin{equation}
  \label{eq:unknownKminComplexity}
  O[(\Tmax/\delta t) \log (\Tmax) \kmax L \log(\kmax L)]
\end{equation}

It should be noted that some Dirichlet-to-Neumann implementations for the Schr\"odinger equation achieve time-complexity $O(\Tmax \log \Tmax)$ \cite{MR2060329,jian:thesis,MR1924419} as well (for compactly supported potentials). Very roughly, where Dirichlet-to-Neumann boundaries require storing the time-history on the boundary, we require storing extra low frequency outgoing waves.

\section{Numerical Results}

In this section we present the results of numerically implementing Algorithm \ref{algo:multiscaleTDPSFPropagator}. The algorithm was implemented in the Python programming language, using the libraries Numarray, Matplotlib and FFTW \cite{FFTW05}. The code is available at {\verb http://cims.nyu.edu/~stucchio } under the Gnu Public License.

\subsection{The Free Schr\"odinger Equation}

We begin by running some comparisons to solutions of the free Schr\"odinger equation (namely \eqref{eq:schroIVP} with $V(x,t)=0$). We solve this equation taking the initial condition $\psi(x,0)=(4 \pi \sigma)^{1/2} e^{i k x } e^{-x^{2}/2\sigma^{2}}$, with $k=1,2,\ldots 21$. With this initial condition, the exact solution is
\begin{equation}
  \psi_{e}(x,t) =
  \frac{
    \exp \left(i k(x-kt/2) \right)
  }{
    \pi^{1/4} 2 \sigma^{1/2} (1+it/\sigma^{2})^{1/2}
    }\exp \left(\frac{-\absSmall{x-kt}^{2}}{2 \sigma^{2}(1+i t/\sigma^{2})} \right).
\end{equation}

We solve \eqref{eq:schro} by means of Algorithm \ref{algo:multiscaleTDPSFPropagator} using 3 scales. Each scale has $1024$ lattice points, taking $\delta x=0.1$ on the finest scale. The parameter $\epsilon$ was taken to be $10^{-8}$, leading to $b'=7.88$ (as per \eqref{eq:defOfFrequencyCutoff}). The maximal resolvable frequency on the finest scale is $k=41.8$. Outgoing waves are filtered from the $n$'th scale at $k=2^{-n} 13.06$. The timestep is taken to be $\delta t = 2^{-5}$.

In each simulation, $\Tmax$ is taken to be $\Tmax=204.8/k$, which is more than enough time for the outgoing wave to reach $x=51.2$ and return to the origin. The quantity
\begin{equation}
  E(k) = \sup_{t \in [0,\Tmax]} \norm{\psi(x,t)-\psi_{e}(x,t)}{L^{2}[-25.6,25.6]}
\end{equation}
was computed, and the result is plotted in Figure \ref{fig:tPlusRTest}.

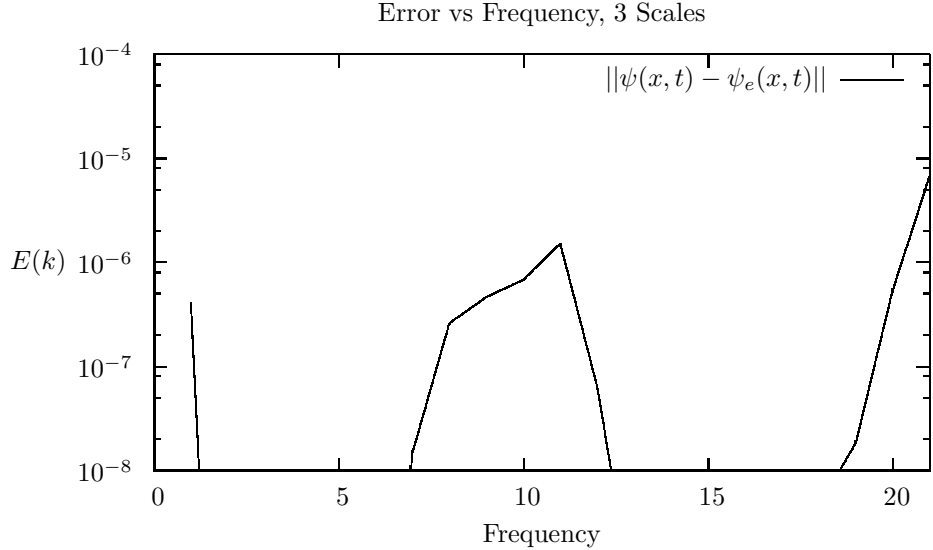
\begin{figure}
\setlength{\unitlength}{0.240900pt}
\ifx\plotpoint\undefined\newsavebox{\plotpoint}\fi
\sbox{\plotpoint}{\rule[-0.200pt]{0.400pt}{0.400pt}}%
\input{free_gaussian_test.tex}
\caption{A graph of error vs the velocity of an outgoing pulse. $\psi_{e}(x,t)$ is the exact solution, $\psi(x,t)$ the numerical one.
}
\label{fig:tPlusRTest}
\end{figure}

Figure \ref{fig:tPlusRTest} shows that the error remains uniformly below $10^{-5}$, for a simulation with $\sigma=4$. The errors will get progressively worse with high frequencies, however, but this can be resolved by increased the rate of sampling.

It should be noted that for $k \approx 1$, if we ran the simulation out to time $t \approx 400+$, we would see errors due to these waves wrapping around the computational domain. This problem is resolved by the use of additional scales.

Another set of simulations was run, this time varying $\sigma$ with $\sigma=1, \ldots, 128$. The maximal time was $\Tmax = 50.0$ in all cases. Note that this initial condition, when $\sigma=2^{5}=64$ is not even localized inside the region $[-51.2,51.2]$, and is therefore completely inaccessible by standard absorbing boundary techniques on a box of size $[-25.6,25.6]$ since the typical wavelength of this solution is longer than the computational domain. The results are plotted in Figure \ref{fig:freeSpreadTest}.

\begin{figure}
\setlength{\unitlength}{0.240900pt}
\ifx\plotpoint\undefined\newsavebox{\plotpoint}\fi
\sbox{\plotpoint}{\rule[-0.200pt]{0.400pt}{0.400pt}}%
\input{free_gaussian_spread_test.tex}
\caption{A graph of error vs the width of an stationary pulse. $\psi_{e}(x,t)$ is the exact solution, $\psi(x,t)$ the numerical one.
}
\label{fig:freeSpreadTest}
\end{figure}
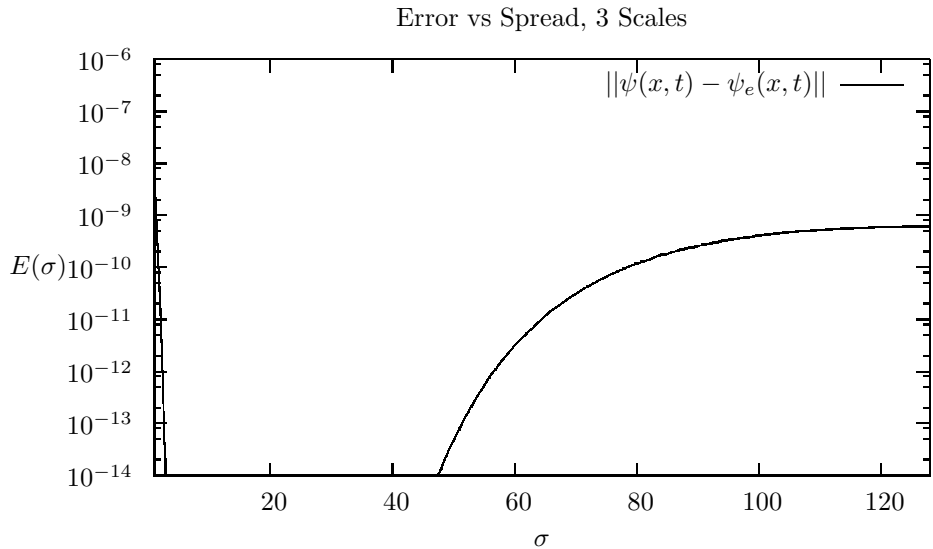

For this simulation, we could have solved the problem with $\sigma=128$ by using a grid with $1024$ lattice points and $\delta x = 0.4$. The reason for using the multiscale TDPSF algorithm is that low and high frequencies can be solved simultaneously. Indeed, tests involving initial conditions containing both wide gaussians (with large $\sigma$) and fast ones (with $k \in [1,15]$) achieved accuracy of $10^{-6}$ as well. This is important for problems where low and high frequencies mix.

\subsection{The Long Range Schr\"odinger Equation}
\label{sec:longRangeTests}

In atomic physics one often considers potentials $V(x,t)$ may be decaying very slowly in $x$. The prototypical case is $V(x)=-Z/\abs{x}$, though we restrict ourselves to the simpler case $V(x) \sim x^{-2}$ near $\pm \infty$:
\begin{equation*}
  V(x) = \frac{-20}{1+25.6^{-2} \abs{x}^{2}} + 20 e^{-x^{2}/9}.
\end{equation*}
The initial data is taken to be $(4\pi)^{-1/4} e^{-x^{2}/2 \cdot 4^{2} }$. The fine grid is taken to occupy the region $[-102.4,102.4]$ with the first filter occupying the region $[-102.4,-51.2] \cup [51.2,102.4]$. The simulation was run with 2, 3 and 4 scales up to a time $t=1500$. To generate an ``exact'' solution to compare to, the same simulation was run using the standard FFT on the large region $[-52428.8,52428.8]$ (requiring $524288$ lattice points). The results are plotted in Figure \ref{fig:multiscalePotentialError}.

To give a picture of the dynamics of the solution,  $\psi(x,t)$ (for various values of $t$) and $V(x)$ are plotted in Figure \ref{fig:psiMovie}. The parameters are $\delta x = 0.1$, $\delta t = 2^{-6}$ and $\delta=10^{-5}$. On the finest scale, gaussians with $\abs{b \ks} \geq 9.1$ are filtered, while others are left on the coarser scales. As can be seen from Figure \ref{fig:psiMovie}, the solution is rough on the interval $[-50,50]$, and high frequencies are present. If fine sampling were not used, aliasing errors would occur. The waves which escape the potential have very low velocity ($v \approx 2$) when they reach $x=50$, too slow to filter by normal methods.

As can be seen from Figure \ref{fig:multiscalePotentialError}, more scales allow the simulation to run for a longer time. The times at which the simulations begin to fail are slightly before $t=200$ (with 2 scales), $t=400$ (with 3 scales) and $t=800$ (with 4 scales). This is in rough agreement with \eqref{eq:coarseScaleErrorPreVirial} and the discussion in Section \ref{sec:interpretingTheError}, although the agreement is not exact due to the presence of the potential (\eqref{eq:23} is valid as written only when $V(x)=0$). The two-scale error occurs more quickly than would be expected since waves move rapidly at the bottom of the potential (by direct analogy to a classical particle).

\begin{figure}
  \setlength{\unitlength}{0.240900pt}
  \ifx\plotpoint\undefined\newsavebox{\plotpoint}\fi
  \sbox{\plotpoint}{\rule[-0.200pt]{0.400pt}{0.400pt}}%
  \includegraphics[scale=0.6]{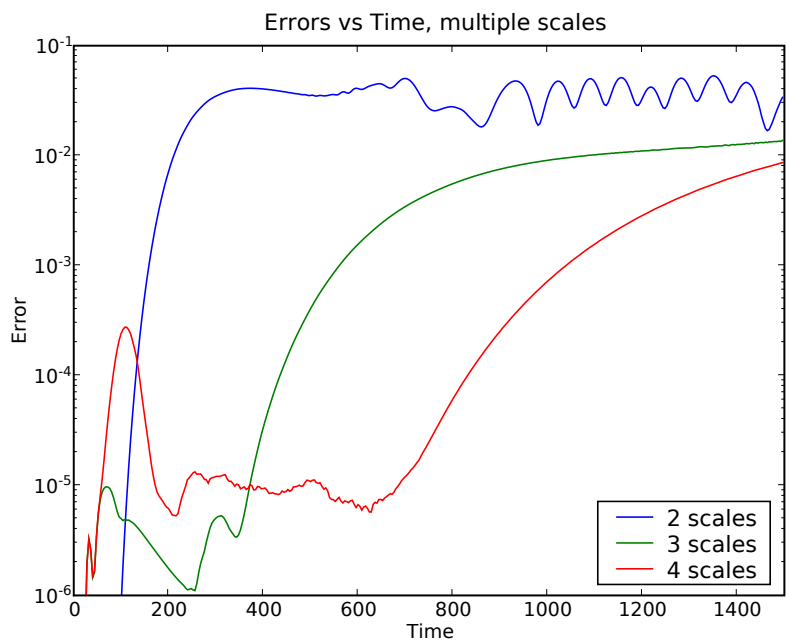}
  \caption{A graph of error (relative error, measured in $L^{2}$) vs time for a test involving a long range potential. As predicted in section \ref{sec:unknownKmin}, the time that the simulation remains accurate increases with the number of scales. }
  \label{fig:multiscalePotentialError}
\end{figure}

\begin{figure}
  \setlength{\unitlength}{0.240900pt}
  \ifx\plotpoint\undefined\newsavebox{\plotpoint}\fi
  \sbox{\plotpoint}{\rule[-0.200pt]{0.400pt}{0.400pt}}%
  \includegraphics[scale=0.6]{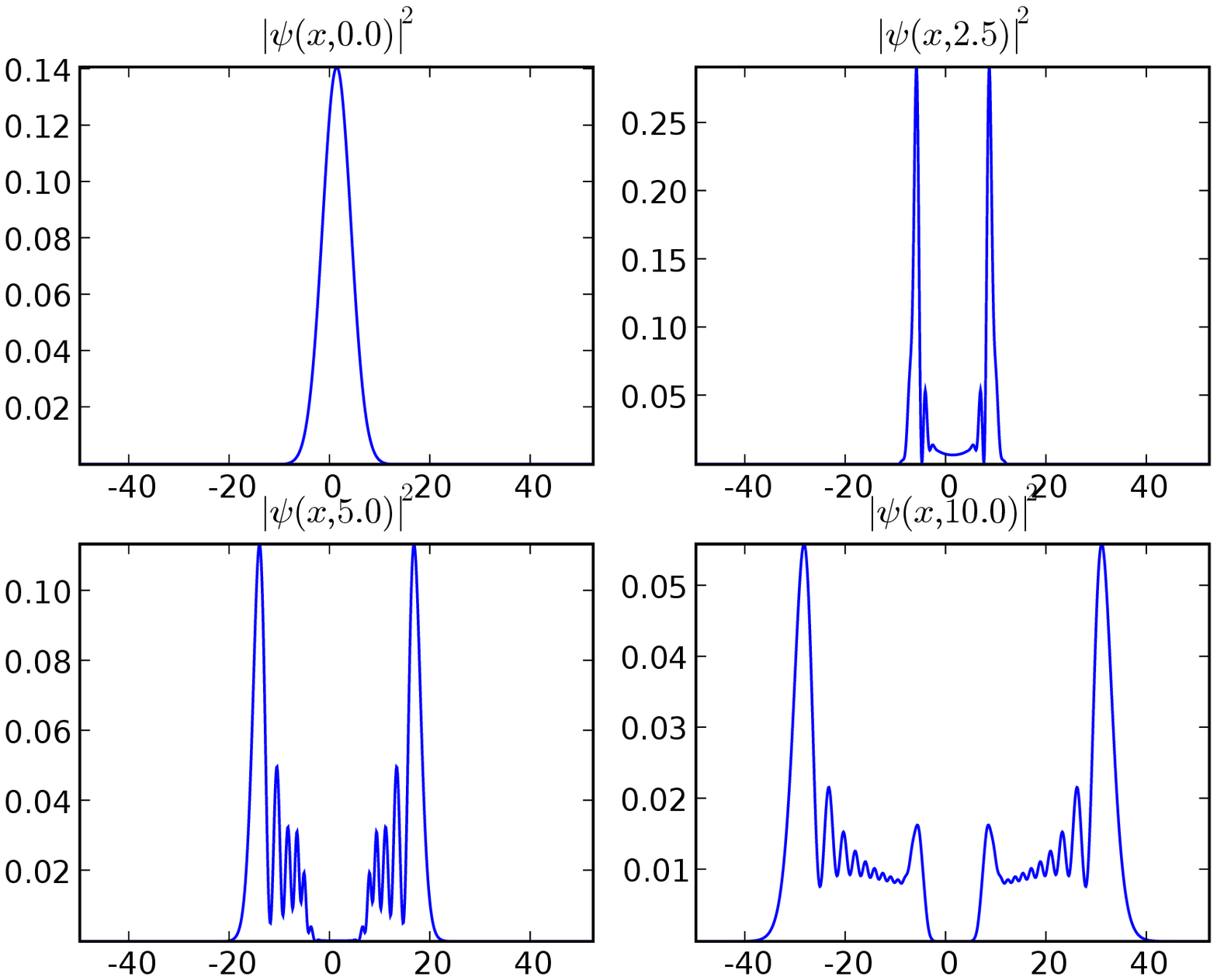}
  \includegraphics[scale=0.6]{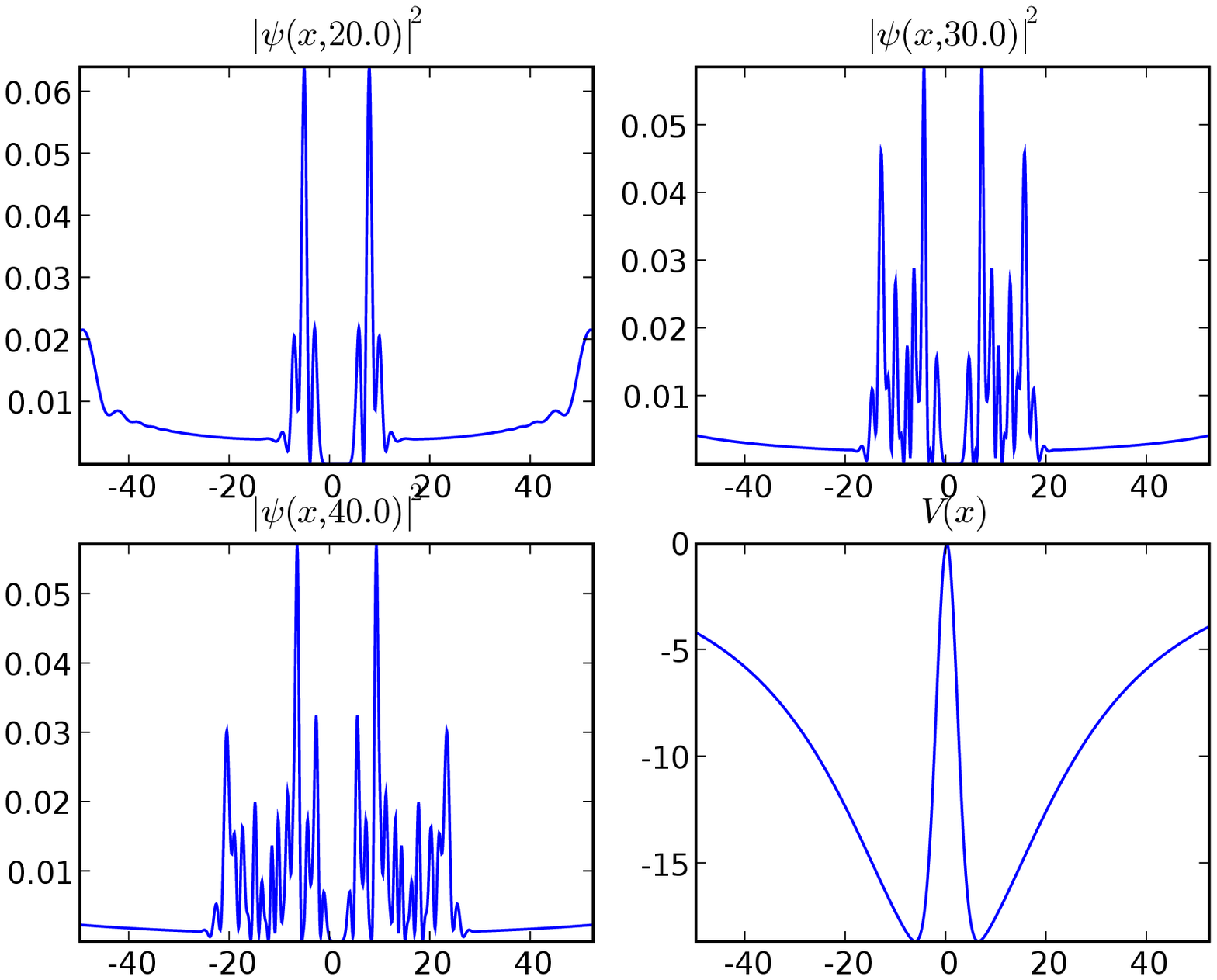}
  \caption{Plots of $\abs{\psi(x,t)}^{2}$ for various times, as well as $V(x)$. For $t \geq 40$, the motion remains confined to the bottom of the potential well.
  }
  \label{fig:psiMovie}
\end{figure}

\section{Conclusion}

In this work we have presented an algorithm for approximating the solution of time-dependent dispersive wave equations on $\Rn$ by domain truncation. All other methods we are aware of attempt to solve the problem on a rectangular region of phase space, which is inaccurate for waves with large wavelengths. We argue that a hyperbolic region $\{ (x,k) : \abs{k} \leq C/\jap{x} \}$ in phase space is the correct region to consider, and provide a spectrally accurate algorithm for approximating differential operators on this region.

Additionally, we extend the Time Dependent Phase Space Filter algorithm \cite{us:TDPSFjcp} to accurately filter outgoing waves regardless of frequency. The computational complexity is proportional to $\log(\kmin^{-1})$ rather than $\kmin^{-1}$ unlike the PML or absorbing potentials.

It should be noted for compactly supported potentials in up to two space dimensions, Dirichlet-to-Neumann boundaries do obtain \cite{MR2060329,jian:thesis,MR1924419} similar $O(\Tmax \log \Tmax)$ complexity (recall Eq. \eqref{eq:unknownKminComplexity}). We are aware of no comparable results for long range potentials, however.

\subsection{Nonlocal Wave Operators and Other Spectral Problems}

The method described here can be applied to other equations, including those for which the wave operator is nonlocal. The wave operator $\omega(k)=\sqrt{k^{2}+\mu^{2}}-\mu$ is an example take from relativistic quantum mechanics (with $\mu > 0$) and quantum field theory (with $\mu=0$). Since $\sqrt{-\Delta+\mu^{2}}-\mu$ is nonlocal, neither finite differences nor finite elements can be used to solve such equations. There is nothing precluding the use of the MTDPSF for these purposes, and preliminary numerical tests suggest Algorithm \ref{algo:multiscaleTDPSFPropagator} works for propagation when $\omega(k)=\sqrt{k^{2}+\mu^{2}}-\mu$.

Additionally, the techniques of Section \ref{sec:multiscaleCalculationDO} can be used for other spectral problems. For instance, we have solved the Poisson equation on $\mathbb{R}$ using Algorithm \ref{algo:multiscaleSchrodingerSolver} with $S(k)=-\Delta^{-1}=k^{2}$ (using the quadrature rule from \cite{demanetSchlagGapCondition} to deal with the singularity at $k=0$) as well as the equation $\sqrt{-\Delta} u =f$. While a new Poisson solver is unremarkable, it is notable that extension to non-local equations is so simple; usually clever tricks are required \cite{MR1856303}.

\subsection{The Missing High Frequencies}

The Multiscale TDPSF algorithm approximates $\psi(x,t)$ only on the interior of $B_{0}$. On the interior of $B_{m}$, the numerical solution approximates $P_{m} \psi(x,t)$. This means that on the vast majority of the computational domain, the information that the algorithm yields is incomplete.

In the broader sense, however, we believe that the TDPSF provides all the relevant information for the problem. The philosophical assumption that underlies open boundaries is the assumption that outside a fixed region of space, the behavior of the solution is free and uninteresting. However, as discussed earlier, this assumption is conceptually flawed with regard to low frequencies. Low frequencies interact over a proportional to their wavelength, which is potentially much greater than the computational box (see Figure \ref{fig:phaseSpaceRegions}).

For this reason, we argue that the MTDPSF algorithm is providing the right information, while all other algorithms are conceptually flawed when low frequencies are present.

\subsection{Future Directions}

When the MTDPSF removes waves from the computational grid, they are removed because they are outgoing and moving nearly freely. In some applications, knowledge of their motion after they leave the domain may still be desired.

This can be rectified at low cost, however, and we intend to investigate this in future works. In the high frequency limit, solutions to \eqref{eq:schroIVP} have simple behavior. A WKB expansion in the time-variable yields a sequence of transport equations which can be safely truncated for frequencies which are high relative to the size of the potential (this corresponds to the absence of turning points).

By the time the TDPSF has removed high frequency waves, they have moved into the WKB regime. This means that if further information concerning their propagation is desired, it can be obtained cheaply by using WKB and then solving a transport equation. Thus, we believe the MTDPSF algorithm can form a building block for solving \eqref{eq:schroIVP} on extremely large regions of space.

\appendix

\section{Proof of Correctness}
\label{sec:ProofOfCorrectness}

In this section our ultimate goal is to prove the correctness of Algorithm \ref{algo:multiscaleSchrodingerSolver}.

We first prove the accuracy of Algorithm \ref{algo:multiscaleFFT} in \ref{sec:proveCorrectAlgo:MultiscaleFFT}. We then use the results there to prove the accuracy of Algorithm \ref{algo:multiscaleDifferentialOperatorCalculation}.

Our main tool is a result describing the accuracy of computing a differential operator by means of a Fourier transform on a finite box. Essentially, the result states that if we restrict a function to a finite box $-[L,L]^{N}$, then the error in approximating $S(i \nabla) \phi(x)$ is the mass of $S(i \nabla) \phi(x)$ located outside $[-L,L]^{N}$. In addition, if we apply a frequency cutoff (such as the one caused by sampling), then there is an additional error corresponding to the mass of $S(i \nabla) \phi(x)$ which is cut off.

\begin{theorem}
  \label{thm:errorControlPSF} Let $S(i\nabla)\varphi (\vec{ x} )$ satisfy the hypothesis of the Poisson summation formula, that is $\abs{S(i\nabla)\varphi(x)} \leq C \jap{x}^{N+\epsilon}$ and $\abs{S(i\vec{k})\hat{\varphi}(k)} \leq C \jap{k}^{N+\epsilon}$. Let $S(\vec{k})$, $S_b(\vec{k})$ be continuous bounded Fourier multiplication operators which are equal for $\vec{k} \in B$ (where $B$ is some closed set).

  Then:
  \begin{multline}
    \label{eq:psfErrorBound}
    \norm{S(i\nabla) \varphi(\vec{x}) - \sum_{\vec{k} \in B}e^{i \pi \vec{k} \cdot \vec{x} / L} S_b(\pi \vec{k}/L) \hat{\varphi}(\pi \vec{k} / L) }{\Hs(B)} \\
    \leq \norm{S(i \nabla) \varphi(\vec{x} + 2 L \vec{n})}{\Hs(([-L,L]^N)^C) }     + \norm{\hat{\varphi}(\vec{k})}{\Hs(B^C)} \sup_{\vec{k} \in B^C} \abs{S(\vec{k})-S_b(\vec{k})}
  \end{multline}
\end{theorem}

\begin{remark}
  A result which is similar to this one appeared in \cite{MR2169959}, where it is used to show that systems on a large enough box can exhibit transient radiative behavior (over short times) in the same way that systems on $\Rn$ can.
\end{remark}

\begin{proof}
  The Poisson summation formula states that:
  \begin{equation}
    \sum_{n \in \mathbbm{Z}^d} f ( \vec{x} + n 2 L ) = \sum_{k \in \mathbbm{Z}^d} e^{i \pi \vec{k} \cdot \vec{x} / L} \hat{f} \left( \pi k/L \right)
    \label{eq:poisson}
  \end{equation}
  We let $\hat{f}(\vec{k})=S(\vec{k}) \hat{\varphi}(\vec{k})$. Then, by rearranging \eqref{eq:poisson}, we find:
  \begin{equation}
    S(i \nabla) \varphi(\vec{x}) - \sum_{\vec{k} \in \Zn}e^{i \pi \vec{k} \cdot \vec{x} / L} S(\pi \vec{k}/L) \hat{\varphi}(\pi \vec{k} / L)
    = - \sum_{\substack{\vec{n} \in \Zn\\ \vec{n} \neq 0}}  S(i \nabla) \varphi(\vec{x} + 2 L \vec{n})
    \label{eq:psfComp}
  \end{equation}
  Since $S(\vec{k})$ and $S_b(\vec{k})$ are equal on $B$, we can add and subtract\\$\sum_{\pi \vec{k}/L \in \Zn} e^{i \pi \vec{k} \cdot \vec{x} / L} (S_b(\pi \vec{k}/L)-S(\pi \vec{k}/L))\hat{\varphi}(\pi \vec{k} / L)$ to both sides of \eqref{eq:psfComp}, to obtain:
  \begin{multline}
    \eqref{eq:psfComp}\\
    = - \sum_{\substack{\vec{n} \in \Zn\\ \vec{n} \neq 0}}  S(i \nabla) \varphi(\vec{x} + 2 L \vec{n}) +
    \sum_{\pi \vec{k}/L \in \Zn} e^{i \pi \vec{k} \cdot \vec{x} / L} (S_b(\pi \vec{k}/L)-S(\pi \vec{k}/L)) \hat{\varphi}(\pi \vec{k} / L)\\
    = \sum_{\vec{n} \in \Zn}  (S_b(i \nabla) - S(i \nabla)) \varphi(\vec{x} + 2 L \vec{n})
  \end{multline}
  where the last line follows by applying \eqref{eq:poisson}. We now take norms and apply the triangle inequality. We find that:
  \begin{equation*}
    \sum_{\substack{\vec{n} \in \Zn\\ \vec{n} \neq 0}} \norm{ S(i \nabla) \varphi(\vec{x} + 2 L \vec{n})}{\Hsb} = \norm{S(i \nabla) \varphi(\vec{x} + 2 L \vec{n})  }{\Hs(([-L,L]^N)^C)}
  \end{equation*}
  and that:
  \begin{multline*}
    \sum_{\vec{n} \in \Zn} \norm{ (S_b(i \nabla) - S(i \nabla)) \varphi(\vec{x} + 2 L \vec{n}) }{\Hsb} = \norm{(S_b(i \nabla) - S(i \nabla)) \varphi(\vec{x} + 2 L \vec{n}) }{\Hs} \\
    \leq \norm{\hat{\varphi}(\vec{k})}{\Hs(B^C)} \sup_{\vec{k} \in B^C} \abs{S(\vec{k})-S_b(\vec{k})}
  \end{multline*}
  We put everything together to obtain the result we seek.
\end{proof}

Using this result, we will show that appropriate differential operators can be reasonably calculated. The main assumption which must be satisfied is that $S(i \nabla)$ can not move $P_{m}(k)\chi_{m}(x)f(x)$ outside $B_{m}$. If $S(i \nabla)$ is a local operator, this property is immediately satisfied.

As a warm-up, we prove the accuracy of Algorithm \ref{algo:spectralInterpolation}.

\begin{theorem}
  Let $f(x)$ be a function such that:
  \begin{equation}
    \int_{\abs{k} > 2^{-m-1}\kmax} \abs{\widehat{\chi_{m-1} f}(k)}^{2} dk \leq \epsilon^{2} \norm{f}{L^{2}}^{2}
  \end{equation}
  Then:
  \begin{equation}
    \norm{\SI{m} f(x) - f(x)}{L^{2}(B_{m-1})} \leq (\epsilon + \delta_{1} )\norm{f}{L^{2}}
  \end{equation}
  Recall $\delta_{1}$ is a desired error bound, and $\chi_{k}$ depends on it (c.f. \eqref{eq:defOfChiPSigma}).
\end{theorem}

\begin{proof}
  Apply Theorem \ref{thm:errorControlPSF} with $S(i \nabla)=1$ to the box $B_{m-1}$. In this case, note that $S_{b}(k) = 1$ for $\abs{k} \leq 2^{-m} \kmax$, and $0$ otherwise. This shows that:
  \begin{equation*}
    \norm{ S(k) \chi_{m-1} f - S_{b}(k) \chi_{m-1} f}{L^{2}(B_{m-1})} \leq \epsilon \norm{f}{L^{2}}
  \end{equation*}
  The result of this is a function supported on $[-2^{-m-1}\kmax, 2^{-m-1}\kmax]$. After padding $2^{-m-1}\kmax \leq \abs{k} \leq 2^{-m}\kmax$ with zeros, the result is a function supported on $[-2^{-m}\kmax, 2^{-m}\kmax]$. Inverse Fourier transforming yields a function $g$ for which $\norm{g - \chi_{m-1} f}{L^{2}(B_{m-1})} \leq \epsilon$. By the definition of $\chi_{m-1}$ (provided \eqref{eq:ConstraintsOnSigmaKmaxL} is satisfied), we find that $\abs{\chi_{m-1} - 1} \leq \delta_{1}$ for $x \in B_{m}$. Thus:
  \begin{multline*}
    \norm{ f - S_{b}(k) \chi_{m-1} f}{L^{2}(B_{m})}\\
    \leq
    \norm{ f - S(k) \chi_{m-1} f}{L^{2}(B_{m})} +
    \norm{ S(k) \chi_{m-1} f - S_{b}(k) \chi_{m-1} f}{L^{2}(B_{m})} \\
    \leq
    \norm{ (1-\chi_{m-1})f}{L^{2}(B_{m-1})} + \epsilon\norm{f}{L^{2}}
    \leq (\delta_{1} + \epsilon )\norm{f}{L^{2}}
  \end{multline*}
  This is what we wanted to show.
\end{proof}

\subsection{Correctness of Algorithm \ref{algo:multiscaleFFT}}
\label{sec:proveCorrectAlgo:MultiscaleFFT}

In this section it is our goal to show that Algorithm \ref{algo:multiscaleFFT} correctly approximates the list of functions $[ \hat{f}^{+}_{0}(k), \hat{f}^{+}_{1}(k), \ldots, \hat{f}^{+}_{M}(k), \hat{f}^{-}_{M}(k)]$. Recall the definition of $\chi_{k}$ and $P_{k}$ and $\sigma$ from \eqref{eq:defOfChiPSigma}.

\begin{lemma}
  \label{lemma:baseCaseDiscretizationEstimate}
  Suppose that $\norm{\hat{f}(k)}{L^{2}([-\kmax,\kmax]^{C})} \leq \delta_{1} \norm{f(x)}{L^{2}}$. Then the function $(1-P_{0}) \chi_{0} f = f_{0}^{k}$ can be accurately approximated on the box $[-L,L]$ with maximal frequency $\kmax$:
  \begin{equation}
    \label{eq:baseCaseDiscretizationEstimate}
    \norm{(1-P_{0}) \chi f - (1-P_{0}^{d}) \chi f}{L^{2}([-L,L])} \leq \delta_{1} \left(2 +\frac{ 8 \kmax}{\sqrt{2\pi}\sigma} \right) \norm{f(x)}{L^{2}}
  \end{equation}
  Here, $P_{0}^{d}$ is the operator $P_{0}$ restricted to the box $[-L,L]$, with a frequency cutoff at $\kmax$.
\end{lemma}

\begin{proof}
  By the aliasing theorem, for $k \leq \kmax$, the discrete version of $P_{0}$ agrees with the continuous version of $P_{0}$, whereas for $k \geq \kmax$, they differ by $2$. Theorem \ref{thm:errorControlPSF} implies therefore that:
  \begin{multline}
    \label{eq:2}
    \norm{(1-P_{0}) \chi f - (1-P_{0})^{d} \chi f}{L^{2}([-L,L])} \\
    \leq     \norm{ \chi f - \chi f}{L^{2}([-L,L])}
    +     \norm{P_{0} \chi f - P_{0}^{d} \chi f}{L^{2}([-L,L])}
    \\
    \leq 0 +  \norm{P_{0} \chi f}{L^{2}([-L,L]^{C})} + 2 \norm{ \chi f}{L^{2}([-\kmax,\kmax]^{C},dk)} \\
    \leq \norm{P_{0} \chi f}{L^{2}([-L,L]^{C})} + 2 \delta_{1} \norm{f(x)}{L^{2}}
  \end{multline}
  The last inequality follows by assumption. Treating the first term requires somewhat more work.

  The operator $P_{0}$ can be written (in the $x$-domain) as a convolution,
  \begin{equation*}
    P_{0}  = \left[ \frac{2}{\sigma\pi^{1/2}} e^{-x^{2}/\sigma^{2}} 2 \kmax \sinc(\kmax x) \right] \star
  \end{equation*}
  where $\sigma$ satisfies \eqref{eq:ConstraintsOnSigmaKmaxL}. Thus, we need to bound:
  \begin{multline}
    \label{eq:1}
    \int_{[-L,L]^{C}} \!\!\!\!\!\!\! dx \abs{ \int_{\mathbb{R}} dx' \left[ \frac{2}{\pi^{1/2}\sigma} e^{-(x-x')^{2}/\sigma^{-2}} 2 \kmax \sinc(\kmax (x-x')) \right] \chi_{0}(x') f(x') }^{2} \\
    \leq 4^{2} \pi^{-1}\sigma^{-2} \kmax^{2} \int_{[-L,L]^{C}} \!\!\!\! dx \abs{ \int_{\mathbb{R}} dx' \left[ e^{-(x-x')^{2}/\sigma^{-2}} \right] \chi_{0}(x') f(x') }^{2} \\
    \leq 4^{2} \pi^{-1}\sigma^{-2} \kmax^{2}
    \int_{\mathbb{R}} \!\!\!\! dx' \chi_{0}(x')^{2} \abs{f(x')}^{2}
    \left[ \int_{[-L,L]^{C}} dx \abs{ e^{-(x-x')^{2}/\sigma^{2}} }^{2} \right] \\
    \leq 4 \kmax^{2} \norm{f(x)}{L^{2}}^{2} \norm{
      \chi_{0}(x') \left[ \int_{[-L,L]^{C}} dx \frac{2}{\sigma \pi^{1/2}} \abs{ e^{-(x-x')^{2}/\sigma^{2}} }^{2} \right]
      }{L^{\infty}}
  \end{multline}
  The second line follows from the first by noting that $\abs{\sinc(z)} \leq 1$ for $z$ real. The function inside the $L^{\infty}$ norm can be explicitly calculated:
  \begin{multline*}
    \frac{1}{2^{1/2} \pi^{1/2} \sigma}
    \left[
      \erf(\sigma^{-1}(3L/4+x')) +  \erf(\sigma^{-1}(x' - 3L/4))
    \right] \\
    \left(
      \erfc(2^{1/2}\sigma^{-1}(L-x')) +  \erfc(2^{1/2}\sigma^{-1}(x' + L))
    \right)
  \end{multline*}
  The quantity inside the square brackets is $2 \chi_{0}(x)$, and is therefore bounded by $ 2\delta_{1}$ for $x' \nin [-5L/6,5L/6]$. For $x' \in [-5L/6,5L/6]$, the quantity inside the $\erfc$ function is at least as large as $\sqrt{2}L/6 \sigma \leq L/3\sigma$; this implies (if $\sigma, L, \kmax$ are chosen according to \eqref{eq:ConstraintsOnSigmaKmaxL}) that the quantity inside the round brackets is bounded by $2 \delta_{1}$. Thus:
  \begin{equation}
    \eqref{eq:1} \leq \delta_{1} \frac{ 8 \kmax^{2}}{\sqrt{2\pi}\sigma} \norm{f(x)}{L^{2}}^{2}
  \end{equation}
  Substituting this bound on \eqref{eq:1} into \eqref{eq:2} yields the result we seek.
\end{proof}

\begin{remark}
  \label{rem:kmaxDueToLaziness}
  We believe the factor of $\kmax$ present in \eqref{eq:baseCaseDiscretizationEstimate} could probably be removed by a more careful analysis. However, the cost of our laziness will only be a logarithmic factor at the end of the day.
\end{remark}

\begin{proposition}
  \label{prop:boundOnfMinussubM}
  We have the bound:
  \begin{equation}
    \norm{f^{-}_{m}(x)-f^{-,d}_{m}(x)}{L^{2}} \leq
    \delta_{1} \left(2m + \frac{16 \kmax }{\sqrt{2\pi} \sigma} \right)\norm{f(x)}{L^{2}}
   \end{equation}
\end{proposition}

\begin{proof}
    Suppose that $\norm{f^{-}_{m}(x)-f^{-,d}_{m}(x)}{L^{2}} \leq E_{m}$. Then simply note that:
  \begin{multline}
    \label{eq:4}
    \norm{f^{-}_{m+1} - f^{-,d}_{m+1}}{L^{2}} = \\
    \norm{[1-\chi_{m}(x)(1-P_{m}) \chi_{m}(x)]f^{-}_{m}(x)
      - [1-\chi_{m}(x)(1-P_{m}^{d}) \chi_{m}(x)] f^{+,d}_{m}(x)}{L^{2}} \\
    \leq \norm{
      \left( [1-\chi_{m}(x)(1-P_{m}) \chi_{m}(x)] -
        [1-\chi_{m}(x)(1-P_{m}^{d}) \chi_{m}(x)] \right)f^{-}_{m}(x)
    }{L^{2}} \\
    + \norm{
      [1-\chi_{m}(x)(1-P_{m}^{d}) \chi_{m}(x)](f^{-,d}_{m}(x)-f^{-}_{m}(x))
    }{L^{2}}
  \end{multline}
  Since $\chi_{m}(x)(1-P_{m}^{d}) \chi_{m}(x)$ is a self-adjoint, positive operator bounded by $1$, we find that:
  \begin{equation*}
    \norm{ 1-\chi_{m}(x)(1-P_{m}^{d}) \chi_{m}(x)}{\mathcal{L}(L^{2},L^{2})} \leq 1
  \end{equation*}
  Thus:
  \begin{multline*}
    \eqref{eq:4} \leq
    \delta_{1} \left(2 +\frac{ 8 \cdot 2^{-m-1}\kmax}{\sqrt{2\pi}2^{m+1}\sigma} \right) \norm{f(x)}{L^{2}} +
    \norm{f^{-,d}_{m}(x)-f^{-}_{m}(x)}{L^{2}} \\
    = \delta_{1} \left(2 +\frac{ 8\cdot 2^{-m-1}\kmax}{\sqrt{2\pi}2^{m+1}\sigma} \right) \norm{f(x)}{L^{2}} + E_{m}
  \end{multline*}
  By summing over the $E_{m}$, (and summing the geometric series $2^{-m}$ from $m=0 \ldots \infty$ to avoid messy dependence on $m$), we obtain the result we seek.
\end{proof}

\begin{proposition}
  \label{prop:boundOnfPlusSubm}
  We have the bound:
  \begin{equation}
    \norm{f^{+}_{m}(x) - f^{+,d}_{m}(x)}{L^{2}} \leq \delta_{1}\left(2m +\frac{16 \kmax }{\sqrt{2\pi} \sigma} \right) \norm{f(x)}{L^{2}}
  \end{equation}
\end{proposition}

\begin{proof}
  A simple calculation, entirely similar to the previous
  \begin{multline*}
    \norm{f^{+}_{m+1} - f^{+,d}_{m+1}}{L^{2}} = \\
    = \norm{\chi_{m}(x)(1-P_{m}) \chi_{m}(x)f^{-}_{m} - \chi_{m}(x)(1-P^{d}_{m}) \chi_{m}(x)f^{-,d}_{m}}{L^{2}} \\
    \leq \norm{\left[\chi_{m}(x)(1-P_{m}) \chi_{m}(x) - \chi_{m}(x)(1-P^{d}_{m}) \chi_{m}(x)\right] f^{-}_{m}}{L^{2}}\\
    + \norm{ \chi_{m}(x)(1-P^{d}_{m}) \chi_{m}(x) \left( f^{-,d}_{m}(x)-f^{-}_{m}(x) \right)}{L^{2}} \\
    \leq \delta_{1} \left(2 +\frac{ 8\cdot 2^{-m-1}\kmax}{\sqrt{2\pi}2^{m+1}\sigma} \right) \norm{f(x)}{L^{2}} +
    E_{m} \\
    \leq \delta_{1} \left(2m + 2 + \frac{16 \kmax }{\sqrt{2\pi} \sigma} \right) \norm{f(x)}{L^{2}}
  \end{multline*}
  with $E_{m}$ given as in the proof of Proposition \ref{prop:boundOnfMinussubM}.
\end{proof}

We have now shown that the effect of discretization on Algorithm \ref{algo:multiscaleFFT} is minimal, and Theorem \ref{thm:accuracyOfMultiscaleFFT} is proven.

\subsection{Correctness of Algorithm \ref{algo:multiscaleDifferentialOperatorCalculation}:  Proof of \eqref{eq:multiscaleDifferentialErrorBound:discretizationErrors}}
\label{sec:proveCorrectAlgo:MultiscaleDifferentialOperatorCalculation}

We prove here the accuracy of Algorithm \ref{algo:multiscaleDifferentialOperatorCalculation}.

\begin{proposition} \dueto{Algorithm \ref{algo:multiscaleDifferentialOperatorCalculation}, Step 2}
  Let $h^{+}_{m}= S(k) f^{+}_{m}$ and let $h^{+,d}_{m}$ be the discrete approximation to $h^{+}_{m}$. Define $h^{-}_{M}$ similarly. Assume also that $\abs{S(k)} \leq 1$. Then the following error bound holds:
  \begin{equation}
    \norm{ h^{+}_{m} - h^{+,d}_{m} }{L^{2}}  \leq \delta_{1}\left(2m +\frac{16 \kmax }{\sqrt{2\pi} \sigma} \right) \norm{f(x)}{L^{2}}
    + \delta_{2} \norm{f(x)}{L^{2}}
  \end{equation}
  Note that this result computes the error associated to Algorithm \ref{algo:multiscaleDifferentialOperatorCalculation} up step 2.
\end{proposition}

\begin{proof}
  Observe that (with $S^{d}(k)$ the discrete approximation to $S(k)$):
  \begin{multline}
    \label{eq:5}
    \norm{ h^{+}_{m} - h^{+,d}_{m} }{L^{2}} = \norm{ S(k) f^{+}_{m} - S^{d}(k) f^{+,d}_{m} }{L^{2}} \\
    \leq \norm{ S(k)( f^{+}_{m} - f^{+,d}_{m})}{L^{2}} + \norm{ [S(k) - S^{d}(k)] f^{+}_{m}}{L^{2}}
  \end{multline}
  Proposition \ref{prop:boundOnfPlusSubm} allows us to bound the first term in the second line of \eqref{eq:5}. The second term is bounded by applying Theorem \ref{thm:errorControlPSF}, and noting that Assumption \ref{ass:differentialOperatorDoesntMoveMuch} bounds the mass localized outside $B_{m}$. Thus, we have:
  \begin{equation*}
    \eqref{eq:5} \leq \delta_{1}\left(2m +\frac{16 \kmax }{\sqrt{2\pi} \sigma} \right) \norm{f(x)}{L^{2}}
    + \delta_{2} \norm{f(x)}{L^{2}}
  \end{equation*}
  This is what we wanted to show.
\end{proof}

Step 3 of Algorithm \ref{algo:multiscaleDifferentialOperatorCalculation} is exact. Since $g^{d}_{m}(x)$ is already band-limited by by discretization, spectral interpolation is exact. This implies immediately the following result:
\begin{multline}
  \label{eq:7}
  \norm{g^{d}_{m}(x) - g_{m}(x)}{L^{2}(B_{m})} \leq \sum_{k=m}^{M} \delta_{1}\left(2k +\frac{16 \kmax }{\sqrt{2\pi} \sigma} \right) \norm{f(x)}{L^{2}} \\
    + \delta_{2} \norm{f(x)}{L^{2}}
    + \norm{S^{d}(k)f^{-,d}_{M}-S(k)f^{-}_{M}(x)}{L^{2}}
    \\
    \leq \delta_{1}\left( 2M(M-m) + (M-m)\frac{16 \kmax }{\sqrt{2\pi} \sigma} \right)\norm{f(x)}{L^{2}}\\
    + \delta_{2}(M-m)\norm{f(x)}{L^{2}} + 2 \norm{f^{-}_{M}(x)}{L^{2}}
\end{multline}

To simplify, we assume that $m=0$ (maximizing the right of \eqref{eq:7}), obtaining:
\begin{multline}
  \label{eq:14}
  \eqref{eq:7} \leq \delta_{1}\left( 2M^{2} + M\frac{16 \kmax }{\sqrt{2\pi} \sigma} \right)\norm{f(x)}{L^{2}}\\
    + \delta_{2}M \norm{f(x)}{L^{2}} + 2 \norm{f^{-}_{M}(x)}{L^{2}}
\end{multline}

The term here which is not controlled by a factor of $\delta_{1,2}$ is $\norm{f^{-}_{M}(x)}{L^{2}}$. This term corresponds to waves on the coarsest scale, and can be thought of roughly as $P_{M} f(x)$. The reason is as follows. By Assumption \ref{ass:highFrequenciesLocalized}, $\chi_{m}(x) P_{m}(k) \chi_{m}(x) f(x) \approx P_{m} f(x)$. This means that each time we subtract $\chi_{m}(x)(1- P_{m}(x)) \chi_{m}(x)$, we are approximating $P_{m}(x)$; the remainder after this is done $M$ times consists solely of low frequencies.

\begin{proposition}
  \label{prop:fminusmApproxPmf}
  We have the following bound:
    \begin{equation}
      \label{eq:boundOnfminusM}
      \norm{ f^{-}_{m} - P_{m} f_{m-1} }{L^{2}} \leq (m+2)\left(3 + \frac{15 \kmax}{4 \cdot 2^{m}}  \right) \delta_{1} \norm{f}{L^{2}}
    \end{equation}
\end{proposition}

We will prove this in Appendix \ref{sec:proofOf:lemma:chiPchiorthogonalToPhaseSpaceLocalizationOperator} (the proof is straightforward but long). The basic idea is that high frequencies of $f(x)$ are localized where $\chi_{m}(x) \approx 1$, so the effect of $\chi_{m}(x)$ is negligible.

This immediately leads to the following result:

\begin{proposition}
  \label{prop:fplusorminusWhatTheyApproximate}
  We have the bounds:
  \begin{subequations}
    \begin{equation}
      \label{eq:approximatefminusm}
      \norm{f^{-}_{m} - \prod_{k=0}^{m} P_{k} f }{L^{2}}
      \leq \delta_{1}\left(
        \frac{3m^{2}+15m}{2} +6+ 75 \kmax
      \right)
      \norm{f}{L^{2}}
    \end{equation}
    \begin{multline}
      \label{eq:approximatefplusm}
      \norm{f^{+}_{m} - (1-P_{m}) \prod_{k=0}^{m-1} P_{k} f }{L^{2}}\\
      \leq \delta_{1}\left( 3m^{2}+15m + 150 \kmax+12 \right)
      \norm{f}{L^{2}}
    \end{multline}
  \end{subequations}
\end{proposition}

\begin{proof}
  First, note that:
  \begin{multline*}
    P_{m} \ldots P_{1} P_{0}f = P_{m} \ldots P_{1} (P_{0}f-f^{-}_{0}) + P_{m} \ldots P_{1}f^{-}_{0}\\
    = P_{1} (P_{0}f-f^{-}_{0}) + P_{m} \ldots P_{2}(P_{1}f^{-}_{0} - f^{-}_{1}) + P_{m} \ldots P_{2}f^{-}_{1} \\
    = \sum_{j=0}^{m}\left[ \left( \prod_{k=j}^{m} P_{k} \right)(P_{k}f^{-}_{k-1} - f^{-}_{k})\right] + f^{-}_{m}
  \end{multline*}
  Bringing $f^{-}_{m}$ to the left, taking norms and using Proposition \ref{prop:fminusmApproxPmf} (as well as the fact that $\norm{P_{m}}{\mathcal{L}(L^{2},L^{2})} \leq 1$, and simple geometric series bounds\footnote{
    Namely the fact that $\sum_{j=0}^{m} 2^{-j} \leq 1/(1-1/2)$ and also $\sum_{j=0}^{m} j 2^{-j} \leq 1/(1-1/2)^{2}$.
    }) shows:

  \begin{multline}
    \label{eq:16}
    \norm{ P_{m} \ldots P_{1} P_{0}f - f^{-}_{m}}{L^{2}} \leq
    \sum_{j=0}^{m} \norm{
      \left[ \left( \prod_{k=j}^{m} P_{k} \right)(P_{k}f^{-}_{k-1} - f^{-}_{k})\right]
      }{L^{2}} \\
      \leq \sum_{j=0}^{m} (j+2)\left(3 + \frac{15 \kmax}{4 \cdot 2^{j}}  \right) \delta_{1} \norm{f}{L^{2}} \\
      \leq \delta_{1}\left(
        \frac{3m(m+1)+12(m+1)}{2} + \sum_{j=0}^{m} \frac{15 \kmax(j+2)}{4\cdot 2^{j}}
      \right) \norm{f}{L^{2}} \\
      \leq \delta_{1}\left(
        \frac{3m^{2}+15m+12}{2} + 15 \kmax + 60 \kmax
        \right) \norm{f}{L^{2}}\\
        =  \delta_{1}\left(
          \frac{3m^{2}+15m}{2} +6+ 75 \kmax
        \right) \norm{f}{L^{2}}
  \end{multline}
  This is \eqref{eq:approximatefminusm}. Finally, note that:
  \begin{multline}
    \label{eq:17}
    \norm{f^{+}_{m} - (1-P_{m}) \prod_{k=0}^{m-1} P_{k} f }{L^{2}} = \\
    \norm{f^{-}_{m-1}-f^{-}_{m} - (P_{m-1} \ldots P_{1} P_{0} f - P_{m} \ldots P_{1} P_{0}f)  }{L^{2}} \\
    \leq \norm{f^{+}_{m} - P_{m-1} \ldots P_{1} P_{0} f}{L^{2}} + \norm{- f^{-}_{m} + P_{m} \ldots P_{1} P_{0}f}{L^{2}}
  \end{multline}
  Using \eqref{eq:16} to bound \eqref{eq:17} yields \eqref{eq:approximatefplusm}.
\end{proof}

We now add and subtract $\prod_{k=0}^{m} P_{k} f$ inside the last norm of \eqref{eq:14}, yielding:
\begin{multline}
  \label{eq:20}
  \eqref{eq:14}  \leq
  \delta_{1}\left( 2M^{2} + M\frac{16 \kmax }{\sqrt{2\pi} \sigma} \right)\norm{f(x)}{L^{2}}\\
    + \delta_{2}M \norm{f(x)}{L^{2}}+ 2 \norm{f^{-}_{M}(x) - \prod_{k=0}^{m} P_{k} f }{L^{2}} + 2\norm{\prod_{k=0}^{m} P_{k} f}{L^{2}}
  \end{multline}
  We apply Proposition \ref{prop:fplusorminusWhatTheyApproximate}, in particular \eqref{eq:approximatefminusm} with $m=M$ to \eqref{eq:20}. We also observe that $\prod_{k=0}^{m} P_{k}$ is a product of commuting, positive operators, and therefore $\norm{\prod_{k=0}^{m} P_{k} f}{} \leq \norm{ P_{M} f}{}$. This yields:
  \begin{multline}
    \label{eq:19}
    \eqref{eq:20}\leq \delta_{1}\left( 2M^{2} + M\frac{16 \kmax }{\sqrt{2\pi} \sigma} \right)\norm{f(x)}{L^{2}} + \delta_{2}M \norm{f(x)}{L^{2}} +\\
    + \delta_{1}\left(
        3M^{2}+15M +12+ 150 \kmax
      \right)
      \norm{f}{L^{2}} + 2\norm{P_{M} f}{L^{2}}
\end{multline}

We now simplify \eqref{eq:19} and formalize it as Proposition \ref{prop:discretizationErrorAlgo:multiscaleDifferentialOperatorCalculation}.

\begin{proposition}
  \label{prop:discretizationErrorAlgo:multiscaleDifferentialOperatorCalculation}
  \dueto{Algorithm \ref{algo:multiscaleDifferentialOperatorCalculation}, Step 3}
  We have the following error bound:
  \begin{multline}
    \label{eq:22}
    \norm{g^{d}_{m}(x) - g_{m}(x)}{L^{2}(B_{m})} \\
    \leq
    \delta_{1} \left( 5M^{2}+\left[  15+\frac{16 \kmax }{\sqrt{2\pi} \sigma} \right]M
    + 12+150\kmax \right)\\
  + \delta_{2} M \norm{f(x)}{L^{2}} + 2\norm{P_{M} f}{L^{2}}
  \end{multline}
  In particular, setting $m=0$ in \eqref{eq:22} recovers \eqref{eq:multiscaleDifferentialErrorBound:discretizationErrors}.
\end{proposition}

\subsection{Correctness of Algorithm \ref{algo:multiscaleDifferentialOperatorCalculation}: Proof of \eqref{eq:multiscaleDifferentialErrorBound:assumptionErrors}}

We have now controlled the discretization errors associated to the calculation of $g^{d}_{m}(x)$, and Theorem \ref{thm:accuracyOf:algo:multiscaleDifferentialOperatorCalculation} is half proved. The only thing remaining is to estimate the difference between $g^{+}_{m}(x)$ and $S(i \partial_{x}) f(x)$ on $B_{m}$, when computed using the exact operators.

First, observe that we can write:
\begin{multline}
  \label{eq:15}
  f(x) = (1-P_{0})f + P_{0} f = (1-P_{0})f + (1-P_{1})P_{0}f + P_{1}P_{0}f \\
  = (1-P_{0})f + (1-P_{1})P_{0}f + (1-P_{2}) P_{1}P_{0}f + P_{2} P_{1}P_{0}f \\
  = (1-P_{0}) f + \left[\sum_{m=1}^{M} (1-P_{m})\left( \prod_{k=0}^{m-1} P_{k} f\right) \right] + \prod_{k=0}^{M} P_{k} f
\end{multline}
Similarly, we can write:
\begin{equation}
  \label{eq:18}
  g_{0}(x) = \left[\sum_{m=0}^{M} S(k)f^{+}_{m} \right] + S(k)f^{-}_{M}
\end{equation}
We wish to compute a bound on $S(k) f(x) - g_{0}(x)$. We will first estimate how close $\sum_{m=0}^{M} f^{+}_{m}$ is to $(1-P_{0}) f + \left[\sum_{m=1}^{M} (1-P_{m})\left( \prod_{k=0}^{m-1} P_{k} f\right) \right]$. Using \eqref{eq:approximatefplusm} term by term, we find:
\begin{multline}
  \norm{\sum_{m=0}^{M} f^{+}_{m}-(1-P_{0}) f + \left[\sum_{m=1}^{M} (1-P_{m})\left( \prod_{k=0}^{m-1} P_{k} f\right) \right]}{L^{2}} \\
  \leq \sum_{m=0}^{M} \delta_{1}\left(
        3m^{2}+15m + 150 \kmax + 12
        \right) \norm{f}{L^{2}} \\
        \leq
        \delta_{1} (
        M^{3}+9M^{2}+(20+150 \kmax)M\\
        +150 \kmax+12  \kmax + 12(M+1)
        ) \norm{f}{L^{2}}
\end{multline}

Combining this with \eqref{eq:15} and \eqref{eq:18}, and using the fact that $S(k)$ is unitary shows that:
\begin{multline}
  \label{eq:21}
  \norm{ S(k) f(x) - g_{0}(x)}{L^{2}} \\
  \leq
  \norm{(1-P_{0}) f + \sum_{m=1}^{M} (1-P_{m})\left( \prod_{k=0}^{m-1} P_{k} f\right) - \sum_{m=0}^{M} f^{+}_{m}}{L^{2}}\\
  + \norm{\prod_{k=0}^{M} P_{k} f - S(k) f^{-}_{M}}{L^{2}} \\
  \leq \delta_{1} (
        M^{3}+9M^{2}+(20+150 \kmax)M
        +150 \kmax+12  \kmax + 12(M+1)
        ) \norm{f}{L^{2}} \\
        + \norm{P_{M} f}{L^{2}} + \norm{S(k) f^{-}_{M}}{L^{2}}
\end{multline}

Using \eqref{eq:approximatefminusm} again, adding and subtracting $\prod_{m=0}^{M} P_{m} f$ inside the last norm of \eqref{eq:21}, we obtain the bound:
\begin{multline}
  \eqref{eq:21} \leq
  \delta_{1} (
  M^{3}+9M^{2}+(20+150 \kmax)M
  +150 \kmax+12  \kmax + 12(M+1)
  ) \norm{f}{L^{2}} \\
  + \norm{P_{M} f}{L^{2}} + \delta_{1}\left(
        \frac{3M^{2}+15M}{2} +6+ 75 \kmax
      \right)
      \norm{f}{L^{2}} + \norm{\prod_{k=0}^{M} P_{k} f}{L^{2}} \\
      \leq \delta_{1} (
      M^{3}+9M^{2}+(20+150 \kmax)M
      +150 \kmax+12  \kmax + 12(M+1)
      ) \norm{f}{L^{2}} \\
      + 2\norm{P_{M} f}{L^{2}} + \delta_{1}\left(
        \frac{3M^{2}+15M}{2} +6+ 75 \kmax
      \right)
      \norm{f}{L^{2}} \\
      \leq \delta_{1} \left(
        M^{3}+10M^{2}+(55/2+\kmax)M+225\kmax+18
      \right)\norm{f}{L^{2}} + 2\norm{P_{M} f}{L^{2}}
\end{multline}

Thus, we have proved \eqref{eq:multiscaleDifferentialErrorBound:assumptionErrors} as well.

\section{Technical Points}

\subsection{Proof of Proposition \ref{prop:fminusmApproxPmf}}
\label{sec:proofOf:lemma:chiPchiorthogonalToPhaseSpaceLocalizationOperator}

We are now almost ready to prove Proposition \ref{prop:fminusmApproxPmf}. We first state a Lemma used in the proof.

\begin{lemma}
  \label{lemma:chiPchiorthogonalToPhaseSpaceLocalizationOperator}
  Provided \eqref{eq:ConstraintsOnSigmaKmaxL} holds, and $m < j$ are integers, we have the following bound:
  \begin{multline}
    \label{eq:assumption2forchiPchi}
    \norm{
      \left[(1-P_{j}(k)) - \chi_{j}(x) (1-P_{j}(k)) \chi_{j}(x) \right] \chi_{m}(x) P_{m}(k) \chi_{m}(x)
      }{\mathcal{L}(L^{2},L^{2})} \\
      \leq \left(1 + \frac{15 \kmax}{4 \cdot 2^{j}}  \right) \delta_{1}
  \end{multline}
\end{lemma}

The idea behind Lemma \ref{lemma:chiPchiorthogonalToPhaseSpaceLocalizationOperator} (proved shortly) is that $\left[(1-P_{j}(k)) - \chi_{j}(x) (1-P_{j}(k)) \chi_{j}(x) \right]$ is ``almost'' supported on $B_{j}^{C}$, while $\chi_{m}(x) P_{m}(k) \chi_{m}(x)$ is supported on $B_{m}$. Since $m < j$, $B_{m} \subset B_{j}$ and so $B_{m} \cap B_{j}^{C}$ is empty. Of course, this is only approximate., so it must be quantified.

\begin{proofof}{Proposition \ref{prop:fminusmApproxPmf}}
  Note that $f^{-}_{m}=[1-\chi_{m}(x)(1- P_{m}(k)) \chi_{m}(x)] f^{-}_{m+1}$. We wish to show that \eqref{eq:approximatefminusm} holds.  Begin by writing:
  \begin{multline}
    \label{eq:12}
    f^{-}_{m} - P_{m}f^{-}_{m-1} =
    [1-\chi_{m}(x) (1-P_{m}(k)) \chi_{m}(x)] f^{-}_{m-1} - P_{m}f^{-}_{m-1}\\
    = [1-P_{m}(k) - \chi_{m}(x) (1-P_{m}(k)) \chi_{m}(x)] f^{-}_{m-1}\\
    = [1-P_{m}(k) - \chi_{m}(x) (1-P_{m}(k)) \chi_{m}(x)](1-\chi_{m-1}(x)(1- P_{m-1}(k))\chi_{m-1}(x)) f^{-}_{m-2} \\
    = [1-P_{m}(k) - \chi_{m}(x) (1-P_{m}(k)) \chi_{m}(x)] f^{-}_{m-2} \\
    - [1-P_{m}(k) - \chi_{m}(x) (1-P_{m}(k)) \chi_{m}(x)] \chi_{m-1}(x)^{2} f^{-}_{m-2} \\
    + [1-P_{m}(k) - \chi_{m}(x) (1-P_{m}(k)) \chi_{m}(x)] \chi_{m-1}(x)P_{m-1}(k)\chi_{m-1}(x) f^{-}_{m-2}
  \end{multline}
  Define $r_{m-j}(x)$ by:
  \begin{multline*}
    r_{m-j}= - [1-P_{m}(k) - \chi_{m}(x) (1-P_{m}(k)) \chi_{m}(x)] \chi_{m-j}(x)^{2} f^{-}_{m-j-1} \\
    + [1-P_{m}(k) - \chi_{m}(x) (1-P_{m}(k)) \chi_{m}(x)] \chi_{m-j}(x)P_{m-j}(k)\chi_{m-j}(x) f^{-}_{m-j-1}
  \end{multline*}
  If we then expand $f^{-}_{m-2}$, and repeat ad-nauseum, \eqref{eq:12} becomes:
  \begin{multline}
    \label{eq:11}
    \eqref{eq:12} = [1-P_{m}(k) - \chi_{m}(x) (1-P_{m}(k)) \chi_{m}(x)] f^{-}_{m-2} + r_{m-1} \\
    = [1-P_{m}(k) - \chi_{m}(x) (1-P_{m}(k)) \chi_{m}(x)](1-\chi_{m-2}(x)(1- P_{m-2}(k))\chi_{m-2}(x)) f^{-}_{m-3}  \\
    + r_{m-1} + r_{m-2} \\
    = \ldots\\
    = [1-P_{m}(k) - \chi_{m}(x) (1-P_{m}(k)) \chi_{m}(x)] f(x) + \sum_{j=0}^{m} r_{m-j}
  \end{multline}
  The term $f(x)$ surfaced since $f^{-}_{0}(x)=(1-\chi_{0}(x)P_{0}(k)\chi_{0}(x))f(x)$. Assume for the moment (it will be proved shortly) that:
  \begin{equation}
    \label{eq:13}
    \norm{r_{m-j}(x)}{L^{2}} \leq \left(3 + (15 \kmax/4) 2^{-m}  \right) \delta_{1} \norm{f}{L^{2}}.
  \end{equation}
  Then:
  \begin{multline}
    \norm{\eqref{eq:11}}{L^{2}} \leq
    \norm{[1-P_{m}(k) - \chi_{m}(x) (1-P_{m}(k)) \chi_{m}(x)] f(x)}{L^{2}}
    + \sum_{j=0}^{m} \norm{r_{m-j}}{L^{2}} \\
    \leq \delta_{1} \norm{f}{L^{2}} + (m+1)\left(3 + \frac{15 \kmax}{4 \cdot 2^{m}}  \right) \delta_{1} \norm{f}{L^{2}} \\
    \leq  (m+2)\left(3 + \frac{15 \kmax}{4 \cdot 2^{m}}  \right) \delta_{1} \norm{f}{L^{2}}
  \end{multline}
  The term $\norm{[1-P_{m}(k) - \chi_{m}(x) (1-P_{m}(k)) \chi_{m}(x)] f(x)}{L^{2}} $ was bounded merely by applying Assumption \ref{ass:highFrequenciesLocalized}. Thus \eqref{eq:boundOnfminusM} is proved, pending a proof of \eqref{eq:13}.

  {\bf Boundedness of $r_{m-j}$: Proof of \eqref{eq:13} }

  We wish to compute a bound on:
  \begin{multline*}
    r_{m-j} =
    - [1-P_{m}(k) - \chi_{m}(x) (1-P_{m}(k)) \chi_{m}(x)] \chi_{m-j}(x)^{2} f^{-}_{m-j-1} \\
    + [1-P_{m}(k) - \chi_{m}(x) (1-P_{m}(k)) \chi_{m}(x)] \chi_{m-j}(x)P_{m-j}(k)\chi_{m-j}(x) f^{-}_{m-j-1}
  \end{multline*}
  By Lemma \ref{lemma:chiPchiorthogonalToPhaseSpaceLocalizationOperator}, we can bound the last term on the right of \eqref{eq:12} by
  \begin{equation*}
    \left(1 + \frac{15 \kmax}{4 \cdot 2^{m}}  \right) \delta_{1} \norm{f^{-}_{m-j-1}}{L^{2}} \leq \left(1 + \frac{15 \kmax}{4 \cdot 2^{m}}  \right) \delta_{1} \norm{f}{L^{2}}.
  \end{equation*}

  The second to last term is bounded as follows. The function $\chi_{m-j}(x)^{2} f^{-}_{m-j-1}$ has $L^{2}$-norm smaller than $\delta_{1}^{2} \norm{f}{L^{2}}$ outside $D = [-2^{m-j}5L/6,2^{m-j}5L/6]$, since $\chi_{m}(x)^{2} < \delta_{1}^{2}$ there. Inside the region $[-2^{m}4L/6,2^{m}4L/6]$, we can write $\chi_{m}(x)=1+h(x)$, where $\abs{h(x)} \leq \delta_{1}^{2}$. Since $D \subset [-2^{m}4L/6,2^{m}4L/6]$, we need only consider this region. Here, we find:
  \begin{multline*}
    \norm{[1-P_{m}(k) - \chi_{m}(x) (1-P_{m}(k)) \chi_{m}(x)] \chi_{m-j}(x)^{2} f^{-}_{m-j-1}}{L^{2}(D)} \\
    = \norm{[1-P_{m}(k) - (1+h(x)) (1-P_{m}(k)) (1+h(x))] \chi_{m-j}(x)^{2} f^{-}_{m-j-1}}{L^{2}(D)} \\
    \leq \norm{[h(x) (1-P_{m}(k)) (1+h(x))]\chi_{m-j}(x)^{2} f^{-}_{m-j-1}}{L^{2}(D)}\\
    + \norm{[(1+h(x)) (1-P_{m}(k)) h(x)] \chi_{m-j}(x)^{2} f^{-}_{m-j-1}}{L^{2}(D)} \\
    \leq 2 \norm{h(x)}{L^{\infty}} \norm{\chi_{m}(x) (1-P_{m}(k))}{\mathcal{L}(L^{2},L^{2})} \norm{\chi_{m-j}(x)^{2} f^{-}_{m-j-1}}{L^{2}(D)} \\
    \leq 2\delta_{1} \norm{f}{L^{2}}
  \end{multline*}
  This completes the proof.
\end{proofof}

Before proving Lemma \ref{lemma:chiPchiorthogonalToPhaseSpaceLocalizationOperator}, we require one additional result to be stated.

\begin{lemma}
  \label{lemma:supportOfGaussianConvolvedChi}
  Let $g(x) = (2/\sqrt{\pi} 2^{j}\sigma) e^{-x^{2}/(2^{j}\sigma)^{2}}$, let $m < j$ be integers. Then:
  \begin{equation}
    \norm{(1-\chi_{j}(x)) g(x) \star \chi_{m}(x) }{\mathcal{L}(L^{2},L^{2})} \leq 5 \delta_{1}
  \end{equation}
\end{lemma}

\begin{proofof}{Lemma \ref{lemma:supportOfGaussianConvolvedChi}}
  Note that:
  \begin{subequations}
    \label{eq:8}
    \begin{eqnarray}
      (1-\chi_{j}(x)) &\leq& 1_{[-2^{j}4L/6,-2^{j}4L/6]^{C}}(x)
      + \delta_{1}1_{[-2^{j}4L/6,-2^{j}4L/6]}(x)\\
      \chi_{m}(x) & \leq & 1_{[-2^{m}5L/6,-2^{m}5L/6]}
      + \delta_{1} 1_{[-2^{m}5L/6,-2^{m}5L/6]^{C}}
    \end{eqnarray}
  \end{subequations}
  Note that the distance between $[-2^{j}4L/6,-2^{j}4L/6]^{C}$ and~$[-2^{m}5L/6,-2^{m}5L/6]$ is at least $2^{j}(4L/6-2^{m-j}5L/6) \geq 2^{j}(4L/6-5L/12)=2^{j}L/4$ (since $m-j\leq 1$).

  We now break $g(x)=g_{1}(x)+g_{2}(x)$, with $g_{1}(x)=g(x)$ for $\abs{x}<2^{j}L/8$ and zero otherwise, and $g_{2}(x)=g(x)-g_{1}(x)$. Simple integration shows that:
  \begin{equation}
    \norm{g_{2}(x)}{L^{1}} \leq 2\erfc\left(
    \frac{
      2^{j}L/8
    }{
      2^{j}\sigma
    }\right) \leq 2 \erfc(L/8\sigma) \leq 2\delta_{1}
  \end{equation}
  This follows since $\sigma > L/8 \erfc^{-1}(\delta_{1})$ and $\erfc(x)$ is monotonically decreasing in $x$. By Young's inequality, this implies that $\norm{ g_{2}(x) \star }{\mathcal{L}(L^{2},L^{2})} \leq 2\delta_{1}$.

  Now, letting $f(x) = (1-\chi_{j}(x)) g(x) \star \chi_{m}(x) h(x)$ (for some $h(x) \in L^{2}$), we observe that:
  \begin{multline}
    \label{eq:423}
    \abs{f(x)} = \int (1-\chi_{j}(x)) g(x-x') \star \chi_{m}(x') h(x') dx' \\
    = \int (1-\chi_{j}(x)) \left[g_{0}(x-x')+g_{1}(x-x') \right] \star \chi_{m}(x') h(x') dx'
  \end{multline}
  Substituting the bounds \eqref{eq:8} into \eqref{eq:423}, we find:
  \begin{multline}
    \eqref{eq:423} \leq
    \left[
      1_{[-2^{j}4L/6,-2^{j}4L/6]^{C}}(x)
      + \delta_{1}1_{[-2^{j}4L/6,-2^{j}4L/6]}(x)
    \right]\\
    \times \left(
      g_{0}(x)\star +g_{1}(x) \star
      \right) \\
      \times \left[
         1_{[-2^{m}5L/6,-2^{m}5L/6]}(x)
      + \delta_{1} 1_{[-2^{m}5L/6,-2^{m}5L/6]^{C}}(x)
      \right] h(x)
  \end{multline}
  Note that \eqref{eq:423} is a product of the form $(A_{1}+\delta_{1}B_{1})(A_{2}+\delta_{1}B_{2})(A_{3}+\delta_{1}B_{3})$. When the sum is expanded, there will be $8$ terms, only one of which does not have a factor of $\delta_{1}$ in it. This term is
  \begin{equation*}
    1_{[-2^{j}4L/6,-2^{j}4L/6]^{C}}(x) \left[ g_{0}(x) \star 1_{[-2^{m}5L/6,-2^{m}5L/6]}(x) h(x)\right]
  \end{equation*}
  But since the support of $g_{0}(x)$ has width $2^{j}L/4$, which is the distance between the support of the first characteristic function and the second, this term must be zero.

There are another three terms of order $\delta_{1}$ (one of which is actually $2\delta_{1}$, the term coming from $g_{1}(x) \star$), and 4 more terms of order $\delta_{1}^{2}$. Provided $\delta_{1} \leq 1/4$, we find that the sum of these terms is $5 \delta_{1}$. This is what we wanted to show.
\end{proofof}

\begin{proofof}{Lemma \ref{lemma:chiPchiorthogonalToPhaseSpaceLocalizationOperator}}
  A calculation. First:
  \begin{multline*}
    \left[(1-P_{j}(k)) - \chi_{j}(x) (1-P_{j}(k)) \chi_{j}(x) \right] \chi_{m}(x) P_{m}(k) \chi_{m}(x)\\
    = (1-\chi_{j}(x))P_{j}(k) \chi_{m}(x) P_{m}(k) \chi_{m}(x) \\
    + \chi_{j}(x) (1-P_{j}(k))(1- \chi_{j}(x)) \chi_{m}(x) P_{m}(k) \chi_{m}(x)
  \end{multline*}
  Recalling \eqref{eq:defOfChi0}, we observe that $\abs{(1- \chi_{j}(x)) \chi_{m}(x)} \leq \delta_{1}$. This implies that:
  \begin{equation*}
    \norm{\chi_{j}(x) (1-P_{j}(k))(1- \chi_{j}(x)) \chi_{m}(x) P_{m}(k) \chi_{m}(x)}{\mathcal{L}(L^{2},L^{2})} \leq \delta_{1}
  \end{equation*}
  Thus we only need to bound $(1-\chi_{j}(x))P_{j}(k) \chi_{m}(x) P_{m}(k) \chi_{m}(x)$. This can be done by examining the integral form of $(1-\chi_{j}(x))P_{j}(k) \chi_{m}(x)$ (the operator $P_{m}(k) \chi_{m}(x)$ has norm bounded by $1$):
  \begin{multline}
    \label{eq:9}
    \abs{ (1-\chi_{j}(x))P_{j}(k) \chi_{m}(x)(\cdot) } \\
    = \abs{ \int  (1-\chi_{j}(x)) e^{-(x-x')^{2}/(2^{j}\sigma)^{2}} \frac{3\kmax}{4 \cdot 2^{j}} \sinc\left( \frac{3\kmax}{8 \cdot 2^{j}} (x-x') \right)
      \chi_{m}(x')
      (\cdot)(x') dx'} \\
    \leq \frac{3\kmax}{4 \cdot 2^{j}} \int  (1-\chi_{j}(x)) e^{-(x-x')^{2}/(2^{j}\sigma)^{2}}
    \chi_{m}(x')
    (\cdot)(x') dx'
  \end{multline}
  The last line follows since $\abs{\sinc(z)} \leq 1$. By Lemma \ref{lemma:supportOfGaussianConvolvedChi}, we find that the norm of the integral operator is bounded by $5 \delta_{1}$, thus:
  \begin{equation}
    \eqref{eq:9} \leq \frac{3 \kmax}{4 \cdot 2^{j}} 5 \delta_{1} \norm{ (\cdot)}{L^{2}}
  \end{equation}
  This is what we wanted to show.
\end{proofof}

\bibliographystyle{plain}
\bibliography{../stucchio}

\end{document}

%% file: free_gaussian_test.tex
\setlength{\unitlength}{0.240900pt}
\ifx\plotpoint\undefined\newsavebox{\plotpoint}\fi
\sbox{\plotpoint}{\rule[-0.200pt]{0.400pt}{0.400pt}}%
\begin{picture}(1500,900)(0,0)
\sbox{\plotpoint}{\rule[-0.200pt]{0.400pt}{0.400pt}}%
\put(221.0,123.0){\rule[-0.200pt]{4.818pt}{0.400pt}}
\put(201,123){\makebox(0,0)[r]{  $ 10^{ -8 } $ }}
\put(1419.0,123.0){\rule[-0.200pt]{4.818pt}{0.400pt}}
\put(221.0,172.0){\rule[-0.200pt]{2.409pt}{0.400pt}}
\put(1429.0,172.0){\rule[-0.200pt]{2.409pt}{0.400pt}}
\put(221.0,237.0){\rule[-0.200pt]{2.409pt}{0.400pt}}
\put(1429.0,237.0){\rule[-0.200pt]{2.409pt}{0.400pt}}
\put(221.0,271.0){\rule[-0.200pt]{2.409pt}{0.400pt}}
\put(1429.0,271.0){\rule[-0.200pt]{2.409pt}{0.400pt}}
\put(221.0,286.0){\rule[-0.200pt]{4.818pt}{0.400pt}}
\put(201,286){\makebox(0,0)[r]{  $ 10^{ -7 } $ }}
\put(1419.0,286.0){\rule[-0.200pt]{4.818pt}{0.400pt}}
\put(221.0,336.0){\rule[-0.200pt]{2.409pt}{0.400pt}}
\put(1429.0,336.0){\rule[-0.200pt]{2.409pt}{0.400pt}}
\put(221.0,401.0){\rule[-0.200pt]{2.409pt}{0.400pt}}
\put(1429.0,401.0){\rule[-0.200pt]{2.409pt}{0.400pt}}
\put(221.0,434.0){\rule[-0.200pt]{2.409pt}{0.400pt}}
\put(1429.0,434.0){\rule[-0.200pt]{2.409pt}{0.400pt}}
\put(221.0,450.0){\rule[-0.200pt]{4.818pt}{0.400pt}}
\put(201,450){\makebox(0,0)[r]{  $ 10^{ -6 } $ }}
\put(1419.0,450.0){\rule[-0.200pt]{4.818pt}{0.400pt}}
\put(221.0,499.0){\rule[-0.200pt]{2.409pt}{0.400pt}}
\put(1429.0,499.0){\rule[-0.200pt]{2.409pt}{0.400pt}}
\put(221.0,564.0){\rule[-0.200pt]{2.409pt}{0.400pt}}
\put(1429.0,564.0){\rule[-0.200pt]{2.409pt}{0.400pt}}
\put(221.0,598.0){\rule[-0.200pt]{2.409pt}{0.400pt}}
\put(1429.0,598.0){\rule[-0.200pt]{2.409pt}{0.400pt}}
\put(221.0,613.0){\rule[-0.200pt]{4.818pt}{0.400pt}}
\put(201,613){\makebox(0,0)[r]{  $ 10^{ -5 } $ }}
\put(1419.0,613.0){\rule[-0.200pt]{4.818pt}{0.400pt}}
\put(221.0,663.0){\rule[-0.200pt]{2.409pt}{0.400pt}}
\put(1429.0,663.0){\rule[-0.200pt]{2.409pt}{0.400pt}}
\put(221.0,728.0){\rule[-0.200pt]{2.409pt}{0.400pt}}
\put(1429.0,728.0){\rule[-0.200pt]{2.409pt}{0.400pt}}
\put(221.0,761.0){\rule[-0.200pt]{2.409pt}{0.400pt}}
\put(1429.0,761.0){\rule[-0.200pt]{2.409pt}{0.400pt}}
\put(221.0,777.0){\rule[-0.200pt]{4.818pt}{0.400pt}}
\put(201,777){\makebox(0,0)[r]{  $ 10^{ -4 } $ }}
\put(1419.0,777.0){\rule[-0.200pt]{4.818pt}{0.400pt}}
\put(221.0,123.0){\rule[-0.200pt]{0.400pt}{4.818pt}}
\put(221,82){\makebox(0,0){ 0}}
\put(221.0,757.0){\rule[-0.200pt]{0.400pt}{4.818pt}}
\put(511.0,123.0){\rule[-0.200pt]{0.400pt}{4.818pt}}
\put(511,82){\makebox(0,0){ 5}}
\put(511.0,757.0){\rule[-0.200pt]{0.400pt}{4.818pt}}
\put(801.0,123.0){\rule[-0.200pt]{0.400pt}{4.818pt}}
\put(801,82){\makebox(0,0){ 10}}
\put(801.0,757.0){\rule[-0.200pt]{0.400pt}{4.818pt}}
\put(1091.0,123.0){\rule[-0.200pt]{0.400pt}{4.818pt}}
\put(1091,82){\makebox(0,0){ 15}}
\put(1091.0,757.0){\rule[-0.200pt]{0.400pt}{4.818pt}}
\put(1381.0,123.0){\rule[-0.200pt]{0.400pt}{4.818pt}}
\put(1381,82){\makebox(0,0){ 20}}
\put(1381.0,757.0){\rule[-0.200pt]{0.400pt}{4.818pt}}
\put(221.0,123.0){\rule[-0.200pt]{293.416pt}{0.400pt}}
\put(1439.0,123.0){\rule[-0.200pt]{0.400pt}{157.549pt}}
\put(221.0,777.0){\rule[-0.200pt]{293.416pt}{0.400pt}}
\put(221.0,123.0){\rule[-0.200pt]{0.400pt}{157.549pt}}
\put(40,450){\makebox(0,0){$E(k)$}}
\put(830,21){\makebox(0,0){Frequency}}
\put(830,839){\makebox(0,0){Error vs Frequency, 3 Scales}}
\put(1279,737){\makebox(0,0)[r]{$||\psi(x,t)-\psi_{e}(x,t)||$}}
\put(1299.0,737.0){\rule[-0.200pt]{24.090pt}{0.400pt}}
\put(1439,585){\usebox{\plotpoint}}
\multiput(1437.92,579.46)(-0.499,-1.549){113}{\rule{0.120pt}{1.334pt}}
\multiput(1438.17,582.23)(-58.000,-176.230){2}{\rule{0.400pt}{0.667pt}}
\multiput(1379.92,398.71)(-0.499,-2.078){113}{\rule{0.120pt}{1.755pt}}
\multiput(1380.17,402.36)(-58.000,-236.357){2}{\rule{0.400pt}{0.878pt}}
\multiput(1321.92,162.73)(-0.497,-0.864){47}{\rule{0.120pt}{0.788pt}}
\multiput(1322.17,164.36)(-25.000,-41.364){2}{\rule{0.400pt}{0.394pt}}
\multiput(937.92,123.00)(-0.496,3.041){41}{\rule{0.120pt}{2.500pt}}
\multiput(938.17,123.00)(-22.000,126.811){2}{\rule{0.400pt}{1.250pt}}
\multiput(915.92,255.00)(-0.499,1.939){113}{\rule{0.120pt}{1.645pt}}
\multiput(916.17,255.00)(-58.000,220.586){2}{\rule{0.400pt}{0.822pt}}
\multiput(856.90,477.92)(-0.508,-0.499){111}{\rule{0.507pt}{0.120pt}}
\multiput(857.95,478.17)(-56.948,-57.000){2}{\rule{0.254pt}{0.400pt}}
\multiput(797.02,420.92)(-1.080,-0.497){51}{\rule{0.959pt}{0.120pt}}
\multiput(799.01,421.17)(-56.009,-27.000){2}{\rule{0.480pt}{0.400pt}}
\multiput(740.24,393.92)(-0.708,-0.498){79}{\rule{0.666pt}{0.120pt}}
\multiput(741.62,394.17)(-56.618,-41.000){2}{\rule{0.333pt}{0.400pt}}
\multiput(683.92,347.77)(-0.499,-1.757){113}{\rule{0.120pt}{1.500pt}}
\multiput(684.17,350.89)(-58.000,-199.887){2}{\rule{0.400pt}{0.750pt}}
\multiput(625.95,135.09)(-0.447,-6.044){3}{\rule{0.108pt}{3.833pt}}
\multiput(626.17,143.04)(-3.000,-20.044){2}{\rule{0.400pt}{1.917pt}}
\multiput(290.92,123.00)(-0.493,10.409){23}{\rule{0.119pt}{8.192pt}}
\multiput(291.17,123.00)(-13.000,245.996){2}{\rule{0.400pt}{4.096pt}}
\put(221.0,123.0){\rule[-0.200pt]{293.416pt}{0.400pt}}
\put(1439.0,123.0){\rule[-0.200pt]{0.400pt}{157.549pt}}
\put(221.0,777.0){\rule[-0.200pt]{293.416pt}{0.400pt}}
\put(221.0,123.0){\rule[-0.200pt]{0.400pt}{157.549pt}}
\end{picture}

%% file: free_gaussian_spread_test.tex
\setlength{\unitlength}{0.240900pt}
\ifx\plotpoint\undefined\newsavebox{\plotpoint}\fi
\sbox{\plotpoint}{\rule[-0.200pt]{0.400pt}{0.400pt}}%
\begin{picture}(1500,900)(0,0)
\sbox{\plotpoint}{\rule[-0.200pt]{0.400pt}{0.400pt}}%
\put(221.0,123.0){\rule[-0.200pt]{4.818pt}{0.400pt}}
\put(201,123){\makebox(0,0)[r]{  $ 10^{ -14 } $ }}
\put(1419.0,123.0){\rule[-0.200pt]{4.818pt}{0.400pt}}
\put(221.0,148.0){\rule[-0.200pt]{2.409pt}{0.400pt}}
\put(1429.0,148.0){\rule[-0.200pt]{2.409pt}{0.400pt}}
\put(221.0,180.0){\rule[-0.200pt]{2.409pt}{0.400pt}}
\put(1429.0,180.0){\rule[-0.200pt]{2.409pt}{0.400pt}}
\put(221.0,197.0){\rule[-0.200pt]{2.409pt}{0.400pt}}
\put(1429.0,197.0){\rule[-0.200pt]{2.409pt}{0.400pt}}
\put(221.0,205.0){\rule[-0.200pt]{4.818pt}{0.400pt}}
\put(201,205){\makebox(0,0)[r]{  $ 10^{ -13 } $ }}
\put(1419.0,205.0){\rule[-0.200pt]{4.818pt}{0.400pt}}
\put(221.0,229.0){\rule[-0.200pt]{2.409pt}{0.400pt}}
\put(1429.0,229.0){\rule[-0.200pt]{2.409pt}{0.400pt}}
\put(221.0,262.0){\rule[-0.200pt]{2.409pt}{0.400pt}}
\put(1429.0,262.0){\rule[-0.200pt]{2.409pt}{0.400pt}}
\put(221.0,279.0){\rule[-0.200pt]{2.409pt}{0.400pt}}
\put(1429.0,279.0){\rule[-0.200pt]{2.409pt}{0.400pt}}
\put(221.0,286.0){\rule[-0.200pt]{4.818pt}{0.400pt}}
\put(201,286){\makebox(0,0)[r]{  $ 10^{ -12 } $ }}
\put(1419.0,286.0){\rule[-0.200pt]{4.818pt}{0.400pt}}
\put(221.0,311.0){\rule[-0.200pt]{2.409pt}{0.400pt}}
\put(1429.0,311.0){\rule[-0.200pt]{2.409pt}{0.400pt}}
\put(221.0,344.0){\rule[-0.200pt]{2.409pt}{0.400pt}}
\put(1429.0,344.0){\rule[-0.200pt]{2.409pt}{0.400pt}}
\put(221.0,360.0){\rule[-0.200pt]{2.409pt}{0.400pt}}
\put(1429.0,360.0){\rule[-0.200pt]{2.409pt}{0.400pt}}
\put(221.0,368.0){\rule[-0.200pt]{4.818pt}{0.400pt}}
\put(201,368){\makebox(0,0)[r]{  $ 10^{ -11 } $ }}
\put(1419.0,368.0){\rule[-0.200pt]{4.818pt}{0.400pt}}
\put(221.0,393.0){\rule[-0.200pt]{2.409pt}{0.400pt}}
\put(1429.0,393.0){\rule[-0.200pt]{2.409pt}{0.400pt}}
\put(221.0,425.0){\rule[-0.200pt]{2.409pt}{0.400pt}}
\put(1429.0,425.0){\rule[-0.200pt]{2.409pt}{0.400pt}}
\put(221.0,442.0){\rule[-0.200pt]{2.409pt}{0.400pt}}
\put(1429.0,442.0){\rule[-0.200pt]{2.409pt}{0.400pt}}
\put(221.0,450.0){\rule[-0.200pt]{4.818pt}{0.400pt}}
\put(201,450){\makebox(0,0)[r]{  $ 10^{ -10 } $ }}
\put(1419.0,450.0){\rule[-0.200pt]{4.818pt}{0.400pt}}
\put(221.0,475.0){\rule[-0.200pt]{2.409pt}{0.400pt}}
\put(1429.0,475.0){\rule[-0.200pt]{2.409pt}{0.400pt}}
\put(221.0,507.0){\rule[-0.200pt]{2.409pt}{0.400pt}}
\put(1429.0,507.0){\rule[-0.200pt]{2.409pt}{0.400pt}}
\put(221.0,524.0){\rule[-0.200pt]{2.409pt}{0.400pt}}
\put(1429.0,524.0){\rule[-0.200pt]{2.409pt}{0.400pt}}
\put(221.0,532.0){\rule[-0.200pt]{4.818pt}{0.400pt}}
\put(201,532){\makebox(0,0)[r]{  $ 10^{ -9 } $ }}
\put(1419.0,532.0){\rule[-0.200pt]{4.818pt}{0.400pt}}
\put(221.0,556.0){\rule[-0.200pt]{2.409pt}{0.400pt}}
\put(1429.0,556.0){\rule[-0.200pt]{2.409pt}{0.400pt}}
\put(221.0,589.0){\rule[-0.200pt]{2.409pt}{0.400pt}}
\put(1429.0,589.0){\rule[-0.200pt]{2.409pt}{0.400pt}}
\put(221.0,606.0){\rule[-0.200pt]{2.409pt}{0.400pt}}
\put(1429.0,606.0){\rule[-0.200pt]{2.409pt}{0.400pt}}
\put(221.0,613.0){\rule[-0.200pt]{4.818pt}{0.400pt}}
\put(201,613){\makebox(0,0)[r]{  $ 10^{ -8 } $ }}
\put(1419.0,613.0){\rule[-0.200pt]{4.818pt}{0.400pt}}
\put(221.0,638.0){\rule[-0.200pt]{2.409pt}{0.400pt}}
\put(1429.0,638.0){\rule[-0.200pt]{2.409pt}{0.400pt}}
\put(221.0,671.0){\rule[-0.200pt]{2.409pt}{0.400pt}}
\put(1429.0,671.0){\rule[-0.200pt]{2.409pt}{0.400pt}}
\put(221.0,687.0){\rule[-0.200pt]{2.409pt}{0.400pt}}
\put(1429.0,687.0){\rule[-0.200pt]{2.409pt}{0.400pt}}
\put(221.0,695.0){\rule[-0.200pt]{4.818pt}{0.400pt}}
\put(201,695){\makebox(0,0)[r]{  $ 10^{ -7 } $ }}
\put(1419.0,695.0){\rule[-0.200pt]{4.818pt}{0.400pt}}
\put(221.0,720.0){\rule[-0.200pt]{2.409pt}{0.400pt}}
\put(1429.0,720.0){\rule[-0.200pt]{2.409pt}{0.400pt}}
\put(221.0,752.0){\rule[-0.200pt]{2.409pt}{0.400pt}}
\put(1429.0,752.0){\rule[-0.200pt]{2.409pt}{0.400pt}}
\put(221.0,769.0){\rule[-0.200pt]{2.409pt}{0.400pt}}
\put(1429.0,769.0){\rule[-0.200pt]{2.409pt}{0.400pt}}
\put(221.0,777.0){\rule[-0.200pt]{4.818pt}{0.400pt}}
\put(201,777){\makebox(0,0)[r]{  $ 10^{ -6 } $ }}
\put(1419.0,777.0){\rule[-0.200pt]{4.818pt}{0.400pt}}
\put(403.0,123.0){\rule[-0.200pt]{0.400pt}{4.818pt}}
\put(403,82){\makebox(0,0){ 20}}
\put(403.0,757.0){\rule[-0.200pt]{0.400pt}{4.818pt}}
\put(595.0,123.0){\rule[-0.200pt]{0.400pt}{4.818pt}}
\put(595,82){\makebox(0,0){ 40}}
\put(595.0,757.0){\rule[-0.200pt]{0.400pt}{4.818pt}}
\put(787.0,123.0){\rule[-0.200pt]{0.400pt}{4.818pt}}
\put(787,82){\makebox(0,0){ 60}}
\put(787.0,757.0){\rule[-0.200pt]{0.400pt}{4.818pt}}
\put(979.0,123.0){\rule[-0.200pt]{0.400pt}{4.818pt}}
\put(979,82){\makebox(0,0){ 80}}
\put(979.0,757.0){\rule[-0.200pt]{0.400pt}{4.818pt}}
\put(1170.0,123.0){\rule[-0.200pt]{0.400pt}{4.818pt}}
\put(1170,82){\makebox(0,0){ 100}}
\put(1170.0,757.0){\rule[-0.200pt]{0.400pt}{4.818pt}}
\put(1362.0,123.0){\rule[-0.200pt]{0.400pt}{4.818pt}}
\put(1362,82){\makebox(0,0){ 120}}
\put(1362.0,757.0){\rule[-0.200pt]{0.400pt}{4.818pt}}
\put(221.0,123.0){\rule[-0.200pt]{293.416pt}{0.400pt}}
\put(1439.0,123.0){\rule[-0.200pt]{0.400pt}{157.549pt}}
\put(221.0,777.0){\rule[-0.200pt]{293.416pt}{0.400pt}}
\put(221.0,123.0){\rule[-0.200pt]{0.400pt}{157.549pt}}
\put(40,450){\makebox(0,0){$E(\sigma)$}}
\put(830,21){\makebox(0,0){$\sigma$}}
\put(830,839){\makebox(0,0){Error vs Spread, 3 Scales}}
\put(1279,737){\makebox(0,0)[r]{$||\psi(x,t)-\psi_{e}(x,t)||$}}
\put(1299.0,737.0){\rule[-0.200pt]{24.090pt}{0.400pt}}
\put(221,631){\usebox{\plotpoint}}
\multiput(221.58,592.23)(0.491,-12.014){17}{\rule{0.118pt}{9.340pt}}
\multiput(220.17,611.61)(10.000,-211.614){2}{\rule{0.400pt}{4.670pt}}
\multiput(231.59,348.48)(0.489,-16.106){15}{\rule{0.118pt}{12.411pt}}
\multiput(230.17,374.24)(9.000,-251.240){2}{\rule{0.400pt}{6.206pt}}
\multiput(668.60,123.00)(0.468,1.358){5}{\rule{0.113pt}{1.100pt}}
\multiput(667.17,123.00)(4.000,7.717){2}{\rule{0.400pt}{0.550pt}}
\put(672,133){\usebox{\plotpoint}}
\multiput(672.59,133.00)(0.489,1.252){15}{\rule{0.118pt}{1.078pt}}
\multiput(671.17,133.00)(9.000,19.763){2}{\rule{0.400pt}{0.539pt}}
\multiput(681.58,155.00)(0.491,1.069){17}{\rule{0.118pt}{0.940pt}}
\multiput(680.17,155.00)(10.000,19.049){2}{\rule{0.400pt}{0.470pt}}
\multiput(691.58,176.00)(0.491,1.017){17}{\rule{0.118pt}{0.900pt}}
\multiput(690.17,176.00)(10.000,18.132){2}{\rule{0.400pt}{0.450pt}}
\multiput(701.59,196.00)(0.489,1.019){15}{\rule{0.118pt}{0.900pt}}
\multiput(700.17,196.00)(9.000,16.132){2}{\rule{0.400pt}{0.450pt}}
\multiput(710.58,214.00)(0.491,0.860){17}{\rule{0.118pt}{0.780pt}}
\multiput(709.17,214.00)(10.000,15.381){2}{\rule{0.400pt}{0.390pt}}
\multiput(720.59,231.00)(0.489,0.902){15}{\rule{0.118pt}{0.811pt}}
\multiput(719.17,231.00)(9.000,14.316){2}{\rule{0.400pt}{0.406pt}}
\multiput(729.58,247.00)(0.491,0.756){17}{\rule{0.118pt}{0.700pt}}
\multiput(728.17,247.00)(10.000,13.547){2}{\rule{0.400pt}{0.350pt}}
\multiput(739.59,262.00)(0.489,0.844){15}{\rule{0.118pt}{0.767pt}}
\multiput(738.17,262.00)(9.000,13.409){2}{\rule{0.400pt}{0.383pt}}
\multiput(748.58,277.00)(0.491,0.652){17}{\rule{0.118pt}{0.620pt}}
\multiput(747.17,277.00)(10.000,11.713){2}{\rule{0.400pt}{0.310pt}}
\multiput(758.58,290.00)(0.491,0.652){17}{\rule{0.118pt}{0.620pt}}
\multiput(757.17,290.00)(10.000,11.713){2}{\rule{0.400pt}{0.310pt}}
\multiput(768.59,303.00)(0.489,0.611){15}{\rule{0.118pt}{0.589pt}}
\multiput(767.17,303.00)(9.000,9.778){2}{\rule{0.400pt}{0.294pt}}
\multiput(777.58,314.00)(0.491,0.600){17}{\rule{0.118pt}{0.580pt}}
\multiput(776.17,314.00)(10.000,10.796){2}{\rule{0.400pt}{0.290pt}}
\multiput(787.59,326.00)(0.489,0.553){15}{\rule{0.118pt}{0.544pt}}
\multiput(786.17,326.00)(9.000,8.870){2}{\rule{0.400pt}{0.272pt}}
\multiput(796.00,336.58)(0.495,0.491){17}{\rule{0.500pt}{0.118pt}}
\multiput(796.00,335.17)(8.962,10.000){2}{\rule{0.250pt}{0.400pt}}
\multiput(806.00,346.59)(0.553,0.489){15}{\rule{0.544pt}{0.118pt}}
\multiput(806.00,345.17)(8.870,9.000){2}{\rule{0.272pt}{0.400pt}}
\multiput(816.00,355.59)(0.495,0.489){15}{\rule{0.500pt}{0.118pt}}
\multiput(816.00,354.17)(7.962,9.000){2}{\rule{0.250pt}{0.400pt}}
\put(825,364){\usebox{\plotpoint}}
\multiput(825.00,364.59)(0.553,0.489){15}{\rule{0.544pt}{0.118pt}}
\multiput(825.00,363.17)(8.870,9.000){2}{\rule{0.272pt}{0.400pt}}
\multiput(835.00,373.59)(0.560,0.488){13}{\rule{0.550pt}{0.117pt}}
\multiput(835.00,372.17)(7.858,8.000){2}{\rule{0.275pt}{0.400pt}}
\multiput(844.00,381.59)(0.721,0.485){11}{\rule{0.671pt}{0.117pt}}
\multiput(844.00,380.17)(8.606,7.000){2}{\rule{0.336pt}{0.400pt}}
\multiput(854.00,388.59)(0.721,0.485){11}{\rule{0.671pt}{0.117pt}}
\multiput(854.00,387.17)(8.606,7.000){2}{\rule{0.336pt}{0.400pt}}
\multiput(864.00,395.59)(0.645,0.485){11}{\rule{0.614pt}{0.117pt}}
\multiput(864.00,394.17)(7.725,7.000){2}{\rule{0.307pt}{0.400pt}}
\multiput(873.00,402.59)(0.852,0.482){9}{\rule{0.767pt}{0.116pt}}
\multiput(873.00,401.17)(8.409,6.000){2}{\rule{0.383pt}{0.400pt}}
\multiput(883.00,408.59)(0.762,0.482){9}{\rule{0.700pt}{0.116pt}}
\multiput(883.00,407.17)(7.547,6.000){2}{\rule{0.350pt}{0.400pt}}
\multiput(892.00,414.59)(0.852,0.482){9}{\rule{0.767pt}{0.116pt}}
\multiput(892.00,413.17)(8.409,6.000){2}{\rule{0.383pt}{0.400pt}}
\multiput(902.00,420.59)(1.044,0.477){7}{\rule{0.900pt}{0.115pt}}
\multiput(902.00,419.17)(8.132,5.000){2}{\rule{0.450pt}{0.400pt}}
\multiput(912.00,425.59)(0.933,0.477){7}{\rule{0.820pt}{0.115pt}}
\multiput(912.00,424.17)(7.298,5.000){2}{\rule{0.410pt}{0.400pt}}
\multiput(921.00,430.59)(1.044,0.477){7}{\rule{0.900pt}{0.115pt}}
\multiput(921.00,429.17)(8.132,5.000){2}{\rule{0.450pt}{0.400pt}}
\multiput(931.00,435.59)(0.933,0.477){7}{\rule{0.820pt}{0.115pt}}
\multiput(931.00,434.17)(7.298,5.000){2}{\rule{0.410pt}{0.400pt}}
\multiput(940.00,440.60)(1.358,0.468){5}{\rule{1.100pt}{0.113pt}}
\multiput(940.00,439.17)(7.717,4.000){2}{\rule{0.550pt}{0.400pt}}
\multiput(950.00,444.60)(1.212,0.468){5}{\rule{1.000pt}{0.113pt}}
\multiput(950.00,443.17)(6.924,4.000){2}{\rule{0.500pt}{0.400pt}}
\multiput(959.00,448.60)(1.358,0.468){5}{\rule{1.100pt}{0.113pt}}
\multiput(959.00,447.17)(7.717,4.000){2}{\rule{0.550pt}{0.400pt}}
\multiput(969.00,452.60)(1.358,0.468){5}{\rule{1.100pt}{0.113pt}}
\multiput(969.00,451.17)(7.717,4.000){2}{\rule{0.550pt}{0.400pt}}
\multiput(979.00,456.61)(1.802,0.447){3}{\rule{1.300pt}{0.108pt}}
\multiput(979.00,455.17)(6.302,3.000){2}{\rule{0.650pt}{0.400pt}}
\multiput(988.00,459.61)(2.025,0.447){3}{\rule{1.433pt}{0.108pt}}
\multiput(988.00,458.17)(7.025,3.000){2}{\rule{0.717pt}{0.400pt}}
\multiput(998.00,462.60)(1.212,0.468){5}{\rule{1.000pt}{0.113pt}}
\multiput(998.00,461.17)(6.924,4.000){2}{\rule{0.500pt}{0.400pt}}
\multiput(1007.00,466.61)(2.025,0.447){3}{\rule{1.433pt}{0.108pt}}
\multiput(1007.00,465.17)(7.025,3.000){2}{\rule{0.717pt}{0.400pt}}
\put(1017,469.17){\rule{2.100pt}{0.400pt}}
\multiput(1017.00,468.17)(5.641,2.000){2}{\rule{1.050pt}{0.400pt}}
\multiput(1027.00,471.61)(1.802,0.447){3}{\rule{1.300pt}{0.108pt}}
\multiput(1027.00,470.17)(6.302,3.000){2}{\rule{0.650pt}{0.400pt}}
\multiput(1036.00,474.61)(2.025,0.447){3}{\rule{1.433pt}{0.108pt}}
\multiput(1036.00,473.17)(7.025,3.000){2}{\rule{0.717pt}{0.400pt}}
\put(1046,477.17){\rule{1.900pt}{0.400pt}}
\multiput(1046.00,476.17)(5.056,2.000){2}{\rule{0.950pt}{0.400pt}}
\put(1055,479.17){\rule{2.100pt}{0.400pt}}
\multiput(1055.00,478.17)(5.641,2.000){2}{\rule{1.050pt}{0.400pt}}
\put(1065,481.17){\rule{2.100pt}{0.400pt}}
\multiput(1065.00,480.17)(5.641,2.000){2}{\rule{1.050pt}{0.400pt}}
\multiput(1075.00,483.61)(1.802,0.447){3}{\rule{1.300pt}{0.108pt}}
\multiput(1075.00,482.17)(6.302,3.000){2}{\rule{0.650pt}{0.400pt}}
\put(1084,485.67){\rule{2.409pt}{0.400pt}}
\multiput(1084.00,485.17)(5.000,1.000){2}{\rule{1.204pt}{0.400pt}}
\put(1094,487.17){\rule{1.900pt}{0.400pt}}
\multiput(1094.00,486.17)(5.056,2.000){2}{\rule{0.950pt}{0.400pt}}
\put(1103,489.17){\rule{2.100pt}{0.400pt}}
\multiput(1103.00,488.17)(5.641,2.000){2}{\rule{1.050pt}{0.400pt}}
\put(1113,491.17){\rule{2.100pt}{0.400pt}}
\multiput(1113.00,490.17)(5.641,2.000){2}{\rule{1.050pt}{0.400pt}}
\put(1123,492.67){\rule{2.168pt}{0.400pt}}
\multiput(1123.00,492.17)(4.500,1.000){2}{\rule{1.084pt}{0.400pt}}
\put(1132,494.17){\rule{2.100pt}{0.400pt}}
\multiput(1132.00,493.17)(5.641,2.000){2}{\rule{1.050pt}{0.400pt}}
\put(1142,495.67){\rule{2.168pt}{0.400pt}}
\multiput(1142.00,495.17)(4.500,1.000){2}{\rule{1.084pt}{0.400pt}}
\put(1151,496.67){\rule{2.409pt}{0.400pt}}
\multiput(1151.00,496.17)(5.000,1.000){2}{\rule{1.204pt}{0.400pt}}
\put(1161,498.17){\rule{1.900pt}{0.400pt}}
\multiput(1161.00,497.17)(5.056,2.000){2}{\rule{0.950pt}{0.400pt}}
\put(1170,499.67){\rule{2.409pt}{0.400pt}}
\multiput(1170.00,499.17)(5.000,1.000){2}{\rule{1.204pt}{0.400pt}}
\put(1180,500.67){\rule{2.409pt}{0.400pt}}
\multiput(1180.00,500.17)(5.000,1.000){2}{\rule{1.204pt}{0.400pt}}
\put(1190,501.67){\rule{2.168pt}{0.400pt}}
\multiput(1190.00,501.17)(4.500,1.000){2}{\rule{1.084pt}{0.400pt}}
\put(1199,502.67){\rule{2.409pt}{0.400pt}}
\multiput(1199.00,502.17)(5.000,1.000){2}{\rule{1.204pt}{0.400pt}}
\put(1209,503.67){\rule{2.168pt}{0.400pt}}
\multiput(1209.00,503.17)(4.500,1.000){2}{\rule{1.084pt}{0.400pt}}
\put(1218,504.67){\rule{2.409pt}{0.400pt}}
\multiput(1218.00,504.17)(5.000,1.000){2}{\rule{1.204pt}{0.400pt}}
\put(1228,505.67){\rule{2.409pt}{0.400pt}}
\multiput(1228.00,505.17)(5.000,1.000){2}{\rule{1.204pt}{0.400pt}}
\put(1247,506.67){\rule{2.409pt}{0.400pt}}
\multiput(1247.00,506.17)(5.000,1.000){2}{\rule{1.204pt}{0.400pt}}
\put(1257,507.67){\rule{2.168pt}{0.400pt}}
\multiput(1257.00,507.17)(4.500,1.000){2}{\rule{1.084pt}{0.400pt}}
\put(1238.0,507.0){\rule[-0.200pt]{2.168pt}{0.400pt}}
\put(1276,508.67){\rule{2.409pt}{0.400pt}}
\multiput(1276.00,508.17)(5.000,1.000){2}{\rule{1.204pt}{0.400pt}}
\put(1266.0,509.0){\rule[-0.200pt]{2.409pt}{0.400pt}}
\put(1295,509.67){\rule{2.409pt}{0.400pt}}
\multiput(1295.00,509.17)(5.000,1.000){2}{\rule{1.204pt}{0.400pt}}
\put(1286.0,510.0){\rule[-0.200pt]{2.168pt}{0.400pt}}
\put(1314,510.67){\rule{2.409pt}{0.400pt}}
\multiput(1314.00,510.17)(5.000,1.000){2}{\rule{1.204pt}{0.400pt}}
\put(1305.0,511.0){\rule[-0.200pt]{2.168pt}{0.400pt}}
\put(1343,511.67){\rule{2.409pt}{0.400pt}}
\multiput(1343.00,511.17)(5.000,1.000){2}{\rule{1.204pt}{0.400pt}}
\put(1324.0,512.0){\rule[-0.200pt]{4.577pt}{0.400pt}}
\put(1391,512.67){\rule{2.409pt}{0.400pt}}
\multiput(1391.00,512.17)(5.000,1.000){2}{\rule{1.204pt}{0.400pt}}
\put(1353.0,513.0){\rule[-0.200pt]{9.154pt}{0.400pt}}
\put(1401.0,514.0){\rule[-0.200pt]{9.154pt}{0.400pt}}
\put(221.0,123.0){\rule[-0.200pt]{293.416pt}{0.400pt}}
\put(1439.0,123.0){\rule[-0.200pt]{0.400pt}{157.549pt}}
\put(221.0,777.0){\rule[-0.200pt]{293.416pt}{0.400pt}}
\put(221.0,123.0){\rule[-0.200pt]{0.400pt}{157.549pt}}
\end{picture}

%% file: multiscale.bbl
\def\cprime{$'$} \def\cprime{$'$} \def\cprime{$'$}
\begin{thebibliography}{10}

\bibitem{abramowitz:handbookmathfunctions}
M.~Abramawitz and I.A. Stegun.
\newblock {\em Handbook of Mathematical Functions}.
\newblock Dover, 1965.

\bibitem{MR1913093}
Bradley Alpert, Leslie Greengard, and Thomas Hagstrom.
\newblock Nonreflecting boundary conditions for the time-dependent wave
  equation.
\newblock {\em J. Comput. Phys.}, 180(1):270--296, 2002.

\bibitem{MR596431}
Alvin Bayliss and Eli Turkel.
\newblock Radiation boundary conditions for wave-like equations.
\newblock {\em Comm. Pure Appl. Math.}, 33(6):707--725, 1980.

\bibitem{MR1294924}
Jean-Pierre Berenger.
\newblock A perfectly matched layer for the absorption of electromagnetic
  waves.
\newblock {\em J. Comput. Phys.}, 114(2):185--200, 1994.

\bibitem{osher:schrodinger}
Li-Tien Chen, Hailiang Liu, and Stanley Osher.
\newblock High-frequency wave propagation in schrodinger equations using the
  level set method.
\newblock {\em Comm. Math. Sci.}, 2006.

\bibitem{colonius:artificialBoundaries}
Tim Colonius.
\newblock Modeling artificial boundary conditions for compressible flow.
\newblock {\em Annu. Rev. Fluid Mech.}, 36:315–45, 2004.

\bibitem{demanetSchlagGapCondition}
L.~Demanet and W.~Schlag.
\newblock Numerical verification of a gap condition for linearized nls.
\newblock {\em Nonlinearity}, 19:829--852, 2006.

\bibitem{weinan:atomisticContinuum}
W.~E and Z.~Huang.
\newblock A dynamic atomistic-continuum method for the simulation of
  crystalline materials.
\newblock {\em J. Comput. Phys.}, 182(1):234--261, 2002.

\bibitem{weinane:multiscaleCrystals}
W.~E and X.-T. Li.
\newblock Multiscale modeling of crystalline solids.
\newblock In {\em Handbook of Computational Material Science}. to appear.

\bibitem{MR0471386}
Bj{{\"o}}rn Engquist and Andrew Majda.
\newblock Absorbing boundary conditions for numerical simulation of waves.
\newblock {\em Proc. Nat. Acad. Sci. U.S.A.}, 74(5):1765--1766, 1977.

\bibitem{MR0436612}
Bjorn Engquist and Andrew Majda.
\newblock Absorbing boundary conditions for the numerical simulation of waves.
\newblock {\em Math. Comp.}, 31(139):629--651, 1977.

\bibitem{MR517938}
Bj{{\"o}}rn Engquist and Andrew Majda.
\newblock Radiation boundary conditions for acoustic and elastic wave
  calculations.
\newblock {\em Comm. Pure Appl. Math.}, 32(3):314--358, 1979.

\bibitem{FFTW05}
Matteo Frigo and Steven~G. Johnson.
\newblock The design and implementation of {FFTW3}.
\newblock {\em Proceedings of the IEEE}, 93(2):216--231, 2005.
\newblock special issue on Program Generation, Optimization, and Platform
  Adaptation.

\bibitem{MR1856303}
Zydrunas Gimbutas, Leslie Greengard, and Michael Minion.
\newblock Coulomb interactions on planar structures: inverting the square root
  of the {L}aplacian.
\newblock {\em SIAM J. Sci. Comput.}, 22(6):2093--2108 (electronic), 2000.

\bibitem{MR1766718}
Leslie Greengard and Patrick Lin.
\newblock Spectral approximation of the free-space heat kernel.
\newblock {\em Appl. Comput. Harmon. Anal.}, 9(1):83--97, 2000.

\bibitem{MR2060329}
Shidong Jiang and L.~Greengard.
\newblock Fast evaluation of nonreflecting boundary conditions for the
  {S}chr{\"o}dinger equation in one dimension.
\newblock {\em Comput. Math. Appl.}, 47(6-7):955--966, 2004.

\bibitem{jian:thesis}
Shidong Jiang and L.~Greengard.
\newblock Efficient representation of nonreflecting boundary conditions for the
  time-dependent schrodinger equation in two dimensions.
\newblock {\em Communications in Pure and Applied Mathematics}, 61(2):261--288,
  2008.

\bibitem{MR1924419}
Christian Lubich and Achim Sch{{\"a}}dle.
\newblock Fast convolution for nonreflecting boundary conditions.
\newblock {\em SIAM J. Sci. Comput.}, 24(1):161--182 (electronic), 2002.

\bibitem{neuhauser:complexPotentials}
Daniel Neuhauser and Micheal Baer.
\newblock The time-dependent schrodinger equation: Application of absorbing
  boundary conditions.
\newblock {\em J. Chem. Phys.}, 90(8), 1989.

\bibitem{MR2169959}
Z.~Rapti, M.~I. Weinstein, and P.~G. Kevrekidis.
\newblock Transient radiative behavior of {H}amiltonian systems in finite
  domains.
\newblock {\em Phys. Lett. A}, 345(1-3):1--9, 2005.

\bibitem{rudd:144104}
Robert~E. Rudd and Jeremy~Q. Broughton.
\newblock Coarse-grained molecular dynamics: Nonlinear finite elements and
  finite temperature.
\newblock {\em Physical Review B (Condensed Matter and Materials Physics)},
  72(14):144104, 2005.

\bibitem{MR1869342}
Achim Sch{{\"a}}dle.
\newblock Non-reflecting boundary conditions for the two-dimensional
  {S}chr{\"o}dinger equation.
\newblock {\em Wave Motion}, 35(2):181--188, 2002.

\bibitem{us:TDPSFrigorous}
A.~Soffer and C.~Stucchio.
\newblock Time dependent phase space filters: Nonreflecting boundaries for
  semilinear schrodinger equations.
\newblock 2006.
\newblock in preparation.

\bibitem{us:TDPSFjcp}
A.~Soffer and C.~Stucchio.
\newblock Open boundaries for the nonlinear schrodinger equation.
\newblock {\em Journal of Computational Physics}, 225(2):1218--1232, 2007.

\bibitem{sogge:fiobook}
Christopher Sogge.
\newblock {\em Fourier Integrals in Classical Analysis}.
\newblock Cambridge University Press, New York, NY, 1993.

\bibitem{stucchio:Thesis:2008}
Chris Stucchio.
\newblock {\em {Selected Problems in Quantum Mechanics}}.
\newblock PhD thesis, Rutgers University, Piscataway, NJ, January 2008.

\bibitem{MR2147901}
J{\'e}r{\'e}mie Szeftel.
\newblock Propagation et r\'eflexion des singularit\'es pour l'\'equation de
  {S}chr{\"o}dinger non lin\'eaire.
\newblock {\em Ann. Inst. Fourier (Grenoble)}, 55(2):573--671, 2005.

\bibitem{szeftel:absorbingBoundaries}
Jeremie Szeftel.
\newblock Absorbing boundary conditions for nonlinear schrodinger equations.
\newblock {\em Numerische Mathematik}, (104):103--127, 2006.

\end{thebibliography}
